\documentclass{amsart}
\usepackage{amsfonts,amscd,amsthm,amsgen,amsmath,amssymb}
\usepackage[all]{xy}
\usepackage{epsfig,color}
\usepackage[vcentermath]{youngtab}
\newtheorem{theorem}{Theorem}[section]
\newtheorem{lemma}[theorem]{Lemma}
\newtheorem{corollary}[theorem]{Corollary}
\newtheorem{proposition}[theorem]{Proposition}

\theoremstyle{definition}
\newtheorem{definition}[theorem]{Definition}
\newtheorem{example}[theorem]{Example}

\theoremstyle{remark}
\newtheorem{remark}[theorem]{Remark}

\numberwithin{equation}{section}

\begin{document}

\setlength\parskip{0.5em plus 0.1em minus 0.2em}

\title{Weights of finite cyclic actions on definite $4$-manifolds}
\author{David Baraglia}

\address{School of Mathematical Sciences, Adelaide University, Adelaide SA 5005, Australia}

\email{david.baraglia@adelaide.edu.au}


\date{\today}

\begin{abstract}
Let $X$ be a closed, smooth, orientable, positive definite $4$-manifold with $H_1(X ; \mathbb{Z}) = 0$. Let $p$ be a prime and suppose that $G = \mathbb{Z}_p$ acts smoothly on $X$ (if $p=2$ we also require an assumption on how $G$ acts on $H^2(X ; \mathbb{Z})$). Using equivariant Yang--Mills theory, Hambleton--Lee and Hambleton--Tanase proved (in the simply-connected case) that the fixed point set and tangential isotropy representations coincide with that of an equivariant connected sum of copies of $\mathbb{CP}^2$ on which $G$ acts linearly. We give a new proof of this result using equivariant Seiberg--Witten theory. Furthermore we also determine the weights of all equivariant line bundles on $X$, showing that these likewise coincide with that of an equivariant connected sum of linear actions on $\mathbb{CP}^2$.
\end{abstract}

\maketitle




\section{Introduction}

Let $X$ be a closed, smooth, orientable, positive definite $4$-manifold with $H_1(X ; \mathbb{Z}) = 0$. Let $p$ be a prime and suppose that $G = \mathbb{Z}_p = \langle g \rangle$ acts smoothly on $X$. If $p=2$ we assume that $H^2(X ; \mathbb{Z})$ contains no cyclotomic summands, by which we mean there does not exist $e \in H^2(X ; \mathbb{Z})$ with $e^2 = 1$ and $ge = -e$. Then the action of $g$ on $H^2(X ; \mathbb{Z})$ is a product of cycles of length $p$. More precisely, there exists an orthonormal basis of $H^2(X ; \mathbb{Z})$ of the form 
\[
e_i, \; 1 \le i \le t, \quad f_{j,k}, \; 1 \le j \le r, 0 \le k \le p-1
\]
such that 
\[
ge_i = e_i, \quad gf_{j,k} = f_{j,k+1}
\]
(with $k$ regarded as an element of $\mathbb{Z}_p$). The fixed point set of $g$ consists of isolated points $x_1, \dots , x_n$ and embedded spheres $\Sigma_1, \dots , \Sigma_m$ where $n+2m = t+2$. At each isolated point $x_i$ the tangent space $T_{x_i}X$ can be identified with $\mathbb{C}^2$ in such a way that the action of $g$ is given by $g(x,y) = ( \omega^{a_i}x , \omega^{b_i}y)$ for some $a_i,b_i \in \mathbb{Z}_p$, where $\omega = e^{2 \pi i/p}$. We call $(a_i,b_i)$ the {\em normal weights} of the point $x_i$. Note that $a_i,b_i$ are only defined up to the following equivalences $(a_i,b_i) \sim (b_i,a_i) \sim (-a_i,-b_i) \sim (-b_i,-a_i)$. Similarly $g$ acts on the normal bundle of each surface $\Sigma_j$ with a normal weight $c_j \in \mathbb{Z}_p$ ($c_j$ is only defined up to $c_j \sim -c_j$, but a choice a orientation on $\Sigma_j$ removes this ambiguity).

In this paper we will study the weights of equivariant line bundles on $X$. Let $L$ be an equivariant line bundle on $X$. Then over each point $x_i$, $g$ acts on $L|_{x_i}$ as multiplication by $\omega^{w_i}$ for some $w_i \in \mathbb{Z}_p$. Similarly over each surface $\Sigma_j$, $g$ acts on $L|_{\Sigma_j}$ as multiplication by $\omega^{v_j}$ for some $v_j \in \mathbb{Z}_p$. We refer to the collections $\{ w_i \}, \{ v_j \}$ as the {\em weights} of the equivariant line bundle $L$.

Recall that equivariant line bundles are classified up to isomorphism by their equivariant first Chern class $c_1(L) \in H^2_{G}(X ; \mathbb{Z})$. In the situation we are considering the forgetful map $H^2_{G}(X ; \mathbb{Z}) \to H^2(X ; \mathbb{Z})^G$ is surjective with kernel $H^2_G(pt ; \mathbb{Z}) = \mathbb{Z}_p$. Put differently, if $L$ is a line bundle on $X$ with $c_1(L) \in H^2(X ; \mathbb{Z})^G$, then $L$ admits a lift of $G$ and this lift is unique up to multiplication by a $p$-th root of unity $\omega^u$.

In this paper we will study and solve the following problem:

{\bf Problem:} For each $c \in H^2(X ; \mathbb{Z})^G$, determine the weights $\{ w_i \} , \{ v_j \}$ of an equivariant line bundle $L$ representing $c$.

Note that the equivariant line bundle $L$ is determined by $c$ only up an element of $H^2_G(pt ; \mathbb{Z})$, so the weights $\{ w_i \}, \{ v_j \}$ are only determined by an overall shift $w_i \mapsto w_i + u$, $v_j \mapsto v_j + u$, where $u \in \mathbb{Z}_p$.

The group $H^2(X ; \mathbb{Z})^G$ is free abelian with basis $e_1, \dots , e_t , f_1 , \dots , f_r$, where $f_j = f_{j,0} + f_{j,1} + \cdots + f_{j,p-1}$. The weights depend linearly on the first Chern class, so to solve the problem we just need to determine the weights for $e_1, \dots , e_t , f_1 , \dots, f_r$. We show in Section \ref{sec:definite} that the classes $f_1, \dots , f_r$ admit lifts to $H^2_G(X ; \mathbb{Z})$ for which the corresponding weights are all zero. The problem reduces to finding the weights for $e_1, \dots , e_t$. Let $\widehat{e}_1, \dots , \widehat{e}_t$ denote (arbitrary) lifts of $e_1, \dots , e_t$ to $H^2_G(X ; \mathbb{Z})$. Let $\{ (w_a)_i\}, \{ (v_a)_j \}$ denote the weights of the equivariant line bundle corresponding to $\widehat{e}_i$. So the problem is to determine $\{ (w_a)_i \}, \{ (v_a)_j \}$.

Our main result says that the fixed point data $a_i,b_i,c_j \in \mathbb{Z}_p$, $[\Sigma_j] \in H^2(X ; \mathbb{Z})$ and the weight data $(w_a)_i , (v_a)_j$ is identical to that of an equivariant connected sum of linear $\mathbb{Z}_p$-actions on $\mathbb{CP}^2$.

\begin{theorem}\label{thm:main}
Let $X$ be a closed, smooth, orientable, positive definite $4$-manifold with $H_1(X ; \mathbb{Z}) = 0$. Let $G = \mathbb{Z}_p = \langle g \rangle$ for $p$ a prime act smoothly on $X$. If $p=2$ assume that $H^2(X ; \mathbb{Z})$ has no cyclotomic summands. Then there is a $4$-manifold $X'$ which is an equivariant connected sum of linear $\mathbb{Z}_p$-actions on $\mathbb{CP}^2$ with the following properties:
\begin{itemize}
\item[(1)]{There is an isometry $\phi : H^2(X' ; \mathbb{Z}) \to H^2(X ; \mathbb{Z})$ which respects the $\mathbb{Z}_p$-actions.}
\item[(2)]{The fixed point sets $\{x_1 , \dots , x_n , \Sigma_1 , \dots , \Sigma_m \}$, $\{x'_1 , \dots , x'_n , \Sigma'_1 , \dots , \Sigma'_m\}$ of $X$ and $X'$ have the same number of points and surfaces and the same normal weights: $(a_i , b_i) = (a'_i , b'_i)$ for all $i$, $c_j = c'_j$ for all $j$. Moreover the fixed surfaces represent the same cohomology classes under $\phi$, that is, $[\Sigma_j] = \phi( [\Sigma'_j])$.}
\item[(3)]{The weights of equivariant line bundles on $X$ and $X'$ agree. More precisely, if we set $e'_i = \phi^{-1}(e_i)$ for $i = 1, \dots , t$, then (for suitably chosen lifts $\widehat{e}_a, \widehat{e}'_a$) we have $(w'_a)_i = (w_a)_i$, $(v'_a)_j = (v_a)_j$ for all $a,i,j$.}
\end{itemize}

\end{theorem}

For an equivariant connected sum of linear $\mathbb{Z}_p$-actions on $\mathbb{CP}^2$, the weights $(w_a)_i, (v_a)_j$ are easily computed. In this sense, Theorem \ref{thm:main} gives an explicit description of the possible weights of equivariant line bundles for $\mathbb{Z}_p$-actions on definite $4$-manifolds.

Parts (1) and (2) of Theorem \ref{thm:main} were orignally proven (in the simply-connected case) by Hambleton--Lee \cite{hl1,hl2} and Hambleton--Tanase \cite{ht} using equivariant Yang-Mills theory. Part (3) of Theorem \ref{thm:main} is a new result which can be seen as a strengthening of the Hambleton--Lee--Tanase result (for cyclic groups of prime order). It says that a $\mathbb{Z}_p$-action on a closed, smooth, positive definite $4$-manifold resembles an equivariant connected sum of linear $\mathbb{Z}_p$-actions on $\mathbb{CP}^2$ in a very strong sense.

We prove Theorem \ref{thm:main} using equivariant Seiberg--Witten theory. In \cite{bar2} we used families Seiberg--Witten theory to deduce a kind of diagonalisation theorem for {\em families} of definite $4$-manifolds. The statement and proof of Theorem \ref{thm:main} are in some sense equivariant counterparts of \cite{bar2}. The proof in \cite{bar2} used a families version of Donaldson's theorem \cite{bar} to deduce properties of the families index of spin$^c$-Dirac operators and from this deduced properties of the tautological classes. In this paper we instead apply an {\em equivariant} version of Donaldson's theorem to equivariant spin$^c$-structures on $X$. This imposes a highly complex system of equations that must be satisfied by the weights of equivariariant line bundles on $X$. We introduce the notion of a {\em weight system} (Definition \ref{def:wtsys}) which captures the abstract structure of these equations. We then show that every abstract weight system takes the form of an equivariant connected sum of linear $\mathbb{Z}_p$-actions on $\mathbb{CP}^2$.

\subsection{Structure of the paper}

The structure of this paper is as follows. In Section \ref{sec:index} we show how the Lefschetz index theorem for the signature and Dirac operators can be used to deduce a series of mod $p$ congruences satisfied by the fixed point data and the weights of determinant line bundles of equivariant spin$^c$-structures. In the case of the signature operator, this techinque is due to Hambeton--Lee--Madsen \cite{hlm}. The case of the Dirac operator appears to be new. In Section \ref{sec:loc} we make use of the localisation theorem in equivariant cohomology to obtain some additional congrences. Furthermore, the localisation perspective allows us to re-write the congruences of Section \ref{sec:index} in a more economical and insightful manner. In Section \ref{sec:definite} we specialise to the case that the $4$-manifold is positive definite. By applying equivariant Seiberg--Witten theory, we obtain a system of equations that the weights of the equivariant line bundles on $X$ must satisfy. We abstract these equations into the notion of a weight system (Definition \ref{def:wtsys}). In Section \ref{sec:decomp} we study abstract weight systems and prove a classification theorem (Theorem \ref{thm:wtdec}) for $p \ge 7$. Every abstract weight system is isomorphic to the weight system of an equivariant connected sum of linear $\mathbb{Z}_p$-actions on $\mathbb{CP}^2$. From this we then deduce Theorem \ref{thm:main} for $p \ge 7$. Finally, in Section \ref{sec:smallp} we deal with the exceptional cases $p = 2,3,5$ where some additional modifications to the definition of a weight system are required in order ensure that the classification theorem and subsequently Theorem \ref{thm:main} carries over.

\noindent{\bf Acknowledgments.} The author was financially supported by an Australian Research Council Future Fellowship, FT230100092.

\section{Index theorems}\label{sec:index}

Let $p$ be a an odd prime. Let $G = \mathbb{Z}_p = \langle g \rangle$ act smoothly on a closed $4$-manifold $X$. Since $p$ is an odd prime, the fixed point set will consits of isolated points $x_1, \dots , x_n$ and orientable surfaces $\Sigma_1, \dots , \Sigma_m$. Set $\omega = e^{2\pi i/p}$. For $u \in \mathbb{Z}_p$, let $\mathbb{C}_u$ denote the $1$-dimensional representation of $G$ where $g$ acts as multiplication by $\omega^u$. We refer to $u$ as the {\em weight} of the representation $\mathbb{C}_u$.

At an isolated point $x_i$ we will have  $T_{x_i} X \cong \mathbb{C}_{a_i} \oplus \mathbb{C}_{b_i}$ for some non-zero weights $a_i,b_i \in \mathbb{Z}_p$. Since $T_{x_i} X$ is a real, oriented representation of $G$ one finds that the weights $(a_i,b_i)$ are only well-defined up ordering $(a_i , b_i ) \sim (b_i , a_i)$ and an overall sign change $(a_i , b_i) \sim (-a_i , -b_i)$.

Let $N_j$ denote the normal bundle of $\Sigma_j$. Since $G$ fixes $\Sigma_j$ pointwise, $G$ will act on $N_j$ with some weight $c_j$. The weight $c_j$ depends on a choice of orientation of $N_j$. If we orient $\Sigma_j$ so that $TX|_{\Sigma_j} \cong T\Sigma_j \oplus N_j$ as oriented vector bundles, then $c_j$ depends on the choice of orientation of $\Sigma_j$. Reversing the orientation on $\Sigma_j$ has the effect of changing $c_j$ to $-c_j$.

Let $Sig$ denote the virtual representation $Sig = H^+(X)_{\mathbb{C}} - H^-(X)_{\mathbb{C}}$ where $H^{\pm}(X)_{\mathbb{C}} = H^{\pm}(X) \otimes_{\mathbb{R}} \mathbb{C}$ and let $\chi_{Sig}(g)$ denote the trace of $g$ acting on $Sig$. The $G$-signature theorem for $X$ (\cite[\textsection 6]{as}) gives 
\[
\chi_{Sig}(g) = \sum_i \frac{(\omega^{a_i} + 1)(\omega^{b_i}+1)}{(\omega^{a_i}-1)(\omega^{b_i} - 1)} - 4 \sum_j \frac{\omega^{c_j}}{(\omega^{c_j}-1)^2} [\Sigma_j]^2.
\]

Let $\sigma_u$ denote the multiplicity of $\mathbb{C}_u$ in $Sig$. Then
\[
\chi_{Sig}(g) = \sum_{u=0}^{p-1} \sigma_u \omega^u.
\]
Note that $\sigma_u = \sigma_{-u}$ because $Sig$ is the complexification of $H^+(X) - H^-(X)$. The $G$-signature theorem now becomes
\begin{equation}\label{equ:gsig1}
\sum_{u=0}^{p-1} \sigma_u \omega^u = \sum_i \frac{(\omega^{a_i} + 1)(\omega^{b_i}+1)}{(\omega^{a_i}-1)(\omega^{b_i} - 1)} - 4 \sum_j \frac{\omega^{c_j}}{(\omega^{c_j}-1)^2} [\Sigma_j]^2.
\end{equation}
This is an equation in the cylotomic field $\mathbb{Q}[\omega] = \mathbb{Q}[t]/(\phi(t))$, where $\phi(t) = 1 + t + \cdots + t^{p-1}$ is the $p$-th cyclotomic polynomial. Following the method used in \cite{hlm}, we will extract a series of mod $p$ congruences from this equation. The first step is to rewrite Equation (\ref{equ:gsig1}) as an equation in $\mathbb{Z}[\omega]$. Given $a \in \mathbb{Z}_p^*$, let $a' \in \mathbb{Z}$ denote the unique integer such that $0 < a' < p$ and $aa' = 1 \; ({\rm mod} \; p)$. We have
\[
\frac{\omega^{a'}-1}{\omega - 1} = 1 + \omega + \omega^2 + \cdots + \omega^{a'-1}.
\]
To this equation we apply the automorphism $\mathbb{Z}[\omega] \to \mathbb{Z}[\omega]$ which sends $\omega$ to $\omega^a$. This gives
\begin{equation}\label{equ:fracwa}
\frac{ \omega - 1}{\omega^{a} - 1} = 1 + \omega^{a} + \omega^{2a} + \cdots + \omega^{(a'-1)a} \in \mathbb{Z}[\omega].
\end{equation}
Equation (\ref{equ:fracwa}) shows that $(\omega - 1)/(\omega^{a} - 1)$ belongs to $\mathbb{Z}[\omega]$. Multiplying both sides of Equation (\ref{equ:gsig1}) by $(\omega-1)^2$ gives:
\begin{align*}
(\omega-1)^2 \sum_{u=0}^{p-1} \sigma_u \omega^u &= \sum_i (\omega^{a_i}+1)(\omega^{b_i} +1) \left(\frac{\omega - 1}{\omega^{a_i} - 1} \right) \left(\frac{\omega - 1}{\omega^{b_i} - 1} \right) \\
& - 4 \sum_j \omega^{c_j} \left(\frac{\omega - 1}{\omega^{c_j} - 1} \right)^2 [\Sigma_j]^2
\end{align*}
which is now an equality in $\mathbb{Z}[\omega] = \mathbb{Z}[t]/(\phi(t))$. We will lifts both sides of this equation to $\mathbb{Z}[t]$. Using Equation (\ref{equ:fracwa}), we see that $(\omega-1)/(\omega^a-1)$ can be lifted to $\psi_a(t)$, where we set
\[
\psi_a(t) = 1 + t^a + t^{2a} + \cdots + t^{(a'-1)a}.
\]
So we have
\begin{align}\label{equ:gsig2}
(t-1)^2 \sum_{u=0}^{p-1} \sigma_u t^u &= \sum_i (t^{a_i}+1)(t^{b_i} +1) \psi_{a_i}(t) \psi_{b_i}(t) \\
&- 4 \sum_j t^{c_j} \psi_{c_j}(t)^2 [\Sigma_j]^2 + f(t)\phi(t) \nonumber
\end{align}
for some $f(t) \in \mathbb{Z}[t]$. Next we will reduce this equation mod $p$. Since $(t-1)\phi(t) = t^p-1 = (t-1)^p \in \mathbb{Z}_p[t]$, it follows that $\psi(t) = (t-1)^{p-1} \in \mathbb{Z}_p[t]$. So reducing mod $p$ and then equating both sides of Equation (\ref{equ:gsig2}) modulo $(t-1)^{p-1}$ gives
\begin{align}\label{equ:gsig3}
(t-1)^2 \sum_{u=0}^{p-1} \sigma_u t^u &= \sum_i (t^{a_i}+1)(t^{b_i} +1) \psi_{a_i}(t) \psi_{b_i}(t) \\
&- 4 \sum_j t^{c_j} \psi_{c_j}(t)^2 [\Sigma_j]^2 \text{ in } \mathbb{Z}_p[t]/( (t-1)^{p-1} ). \nonumber
\end{align}

Set $x = t-1$ so that $\mathbb{Z}_p[t]/( (t-1)^{p-1} ) = \mathbb{Z}_p[x]/(x^{p-1})$. We will expand both sides of Equation (\ref{equ:gsig3}) in powers of $x$ and then equate coefficients up to order $p-2$. Actually we will only expand up to order $x^4$ when $p \ge 7$, to order $x^3$ when $p=5$ and to order $x$ when $p=3$. For this we need the expansions of $t^u$ and $\psi_a(t)$ in powers of $x$. Clearly,
\[
t^u = (1+x)^u = 1 + ux + \binom{u}{2}x^2 + \binom{u}{3}x^3 + \cdots
\]
and hence
\begin{align*}
\psi_a(t) &= \sum_{r=0}^{a'-1} t^{ra} \\
&= \sum_{r=0}^{a'-1} \sum_{k \ge 0} \binom{ra}{k} x^k \\
&= \sum_{k \ge 0} c(a,k) x^k
\end{align*}
where
\[
c(a,k) = \sum_{r=0}^{a'-1} \binom{ra}{k}.
\]
We have $c(a,0) = a'$ and
\[
c(a,1) = \sum_{r=0}^{a'-1} ar = aa' (a'-1)/2 = (a'-1)/2 = -(a-1)/2a \; ({\rm mod} \; p).
\]
In a similar manner we can compute $c(a,2), \dots , c(a,4) \; ({\rm mod} \; p)$. This gives
\begin{align}\label{equ:psia}
\psi_a(t) &= \frac{1}{a} - \frac{(a-1)}{2a}x + \frac{(a^2-1)}{12a}x^2 - \frac{(a^2-1)}{24a}x^3 \\ 
& \quad \quad - \frac{(a^2-1)(a^2-19)}{720a}x^4 \text{ in } \mathbb{Z}_p[x]/(x^5) \text{ for } p \ge 7. \nonumber
\end{align}

Notice that the denominators only have $2,3,5$ in their prime factorisation, so this equation makes sense in $\mathbb{Z}_p[x]/(x^5)$ for $p \ge 7$. When $p=5$ the same equation holds up to order $x^3$ and when $p=3$ the equation holds up to order $x$. From this, we find

\begin{align*}
(t^a+1)\psi_a(t) & = \frac{2}{a} + \frac{1}{a}x + \frac{(a^2-1)}{6a}x^2 - \frac{(a^2-1)}{12a}x^3 \\
& \quad \quad - \frac{(a^2-1)(a^2-19)}{360a}x^4 +  \text{ in } \mathbb{Z}_p[x]/(x^5),
\end{align*}

which holds up to order $x^3$ when $p=5$ and up to $x$ when $p=3$. Unless stated otherwise all of the expansions given below will hold only or order $x^3$ when $p=5$ and to order $x$ when $p=3$. Taking the above equation for $a$ and $b$ and multiplying them together gives

\begin{align*}
(t^a+1)(t^b+1)\psi_a(t)\psi_b(t) &= \frac{4}{ab} + \frac{4}{ab}x + \frac{1}{3}\left(\frac{a^2+b^2+1}{ab}\right)x^2 \\
& \quad \quad -\frac{1}{180}\left( \frac{ a^4 + b^4 -5a^2b^2 + 3}{ab} \right)x^4 \text{ in } \mathbb{Z}_p[x]/(x^5).
\end{align*}

Setting Equation (\ref{equ:psia}) and replacing $a$ by $c$ gives

\begin{align*}
\psi_c(t)^2 &= \frac{1}{c^2}  - \frac{(c-1)}{c^2}x + \frac{(5c-1)(c-1)}{12c^2}x^2 \\
& \quad \quad -\frac{(c^3-c)}{12c^2}x^3 + \frac{(c^2-1)(c^2+10c+1)}{240c^2}x^4 \text{ in } \mathbb{Z}_p[x]/(x^5).
\end{align*}

To compute $t^c \psi_c(t)^2$, first note that $(t^c-1)\psi_c(t) = t-1 = x \in \mathbb{Z}_p[x]/(x^5)$. So $(t^c-1)\psi_c(t)^2 = x \psi_c(t)$ and thus
\[
t^c \psi_c(t)^2 = x \psi_c(t) + \psi_c(t)^2 \text{ in } \mathbb{Z}_p[x]/(x^5).
\]
So from the expansions of $\psi_c(t)$ and $\psi_c(t)^2$, we obtain
\[
t^c \psi_c(t)^2 = \frac{1}{c^2} +\frac{1}{c^2}x - \frac{(c^2-1)}{12c}x^2 + \frac{(c^4-1)}{240c^2}x^4 \text{ in } \mathbb{Z}_p[x]/(x^5).
\]

Notice that
\[
x^2(t^u + t^{p-u}) = 2x^2 + u^2 x^4 \text{ in } \mathbb{Z}_p[x]/(x^5).
\]
Since $\sigma_u = \sigma_{-u}$, it follows that
\[
x^2 \sum_{u=0}^{p-1} \sigma_u t^u = x^2 \sigma(X) + \sum_{u=1}^{p-1} \frac{1}{2}u^2 x^4 \text{ in } \mathbb{Z}_p[x]/(x^5).
\]

After substituting these expansions into Equation (\ref{equ:gsig3}) and equating coefficients, we obtain

\begin{align}
\sum_i \frac{1}{a_i b_i} - \sum_j \frac{[\Sigma_j]^2}{c_j^2} &= 0, \label{equ:gsigp1} \\
\sum_i \frac{a_i^2 + b_i^2}{a_ib_i} + \sum_j [\Sigma_j]^2 &= 3\sigma(X), \label{equ:gsigp2} \\
\sum_i \frac{ a_i^4 + b_i^4 - 5a_i^2 b_i^2}{a_ib_i} + 3 \sum_j [\Sigma_j]^2 c_j^2 &= -90 \sum_u u^2 \sigma_u. \label{equ:gsigp3}
\end{align}
Equation (\ref{equ:gsigp3}) is valid for $p \ge 7$, Equations (\ref{equ:gsigp2}) is valid for $p \ge 5$ and Equation (\ref{equ:gsigp1}) is valid for all odd $p$. We will see in Section \ref{sec:loc} that (\ref{equ:gsigp2}) is in fact also valid for $p=3$.

Next we consider the Lefschetz index for spin$^c$-Dirac operators \cite{as}. To apply the Lefschetz index theorem we need equivariant spin$^c$-structures. 

\begin{proposition}\label{prop:eqspinc}
Let $X$ be a closed orientable smooth $4$-manifold with $b_1(X) = 0$ and let $G = \mathbb{Z}_p$ for a prime $p$ act smoothly and orientation preservingly on $X$. Let $\mathfrak{s}$ be a spin$^c$-structure that is preserved by $G$. Then $\mathfrak{s}$ can be made into a $G$-equivariant spin$^c$-structure. If $p$ is odd, then for every lift $\widehat{c(\mathfrak{s})}$ of $c(\mathfrak{s})$ to $H^2_G(X ; \mathbb{Z})$, $\mathfrak{s}$ is the equivariant first Chern class of the determinant line bundle of a lift of $G$ to $\mathfrak{s}$. If $p=2$, then one and only one lift of $c(\mathfrak{s})$ to $H^2_G(X ; \mathbb{Z})$ is the equivariant first Chern class of the determinant line bundle of a lift of $G$ to $\mathfrak{s}$.
\end{proposition}
\begin{proof}
Let $\mathfrak{s}$ be a $G$-invariant spin$^c$-structure. Using the same method as \cite[\textsection 3]{bh}, we obtain an $S^1$-central extension $1 \to S^1 \to G_{\mathfrak{s}} \to G \to 1$ of $G$ and a lift of the $G$-action on $X$ to an action of $G_{\mathfrak{s}}$ on the spinor bundles corresponding to $\mathfrak{s}$, such that the $S^1$ subgroup acts as scalar multiplication. Let $g \in G$ be a generator of $G$ and let $g' \in G_{\mathfrak{s}}$ be a lift of $g$ to $G_{\mathfrak{s}}$. We have $(g')^p = \zeta$ for some $\zeta \in S^1$. Let $u \in S^1$ satisfy $u^p = \zeta$. Then $\widehat{g} = u^{-1}g'$ satisfies $(\widehat{g})^p = 1$. Hence $\widehat{g}$ determines a lift of $G$ to the spinor bundles associated to $\mathfrak{s}$. Different lifts of $g$ are given by $\omega^k \widehat{g}$ where $\omega = e^{2\pi i/p}$. 

Let $c_0 \in H^2_G(X ; \mathbb{Z})$ denotes the equivariant first Chern class of the determinant line bundle of $\mathfrak{s}$ using the lift $\widehat{g}$. Then $c_0$ is a lift of $c(\mathfrak{s})$ to equviariant cohomology. View $\mathbb{C}_1$ as an equivariant line bundle over $\{pt\}$ and let $x \in H^2_G(pt ; \mathbb{Z})$ be the equivariant first Chern class of $\mathbb{C}_1$. Thus $x$ is a generator of $H^2_G(pt ; \mathbb{Z})$. Since $b_1(X) = 0$, it follows from the Borel spectral sequence that any lift of $c(\mathfrak{s})$ to equivariant cohomology is of the form $c_0 + mx$ for some $m \in \mathbb{Z}_p$. The equivariant first Chern class of the determinant line bundle of $\mathfrak{s}$ using the lift $\omega^k \widehat{g}$ is $c_0 + 2k x$. If $p$ is odd, then every class $c_0 + m x$ can be written as $c_0 + 2k x$ for suitable $k$. Hence every lift of $c(\mathfrak{s})$ to equivariant cohomology is realised by some lift of $G$ to $\mathfrak{s}$. If $p=2$ then $c_0 + 2kx = c_0$ for every choice of lift.
\end{proof}

Let $\mathfrak{s}$ be a $G$-equivariant spin$^c$-structure. Over each connected component of the fixed point set, $g$ acts on $L$, the determinant line of $\mathfrak{s}$ with some weight which we denote by $w_i$ for an isolated point $x_i$ and $v_j$ for a fixed surface $\Sigma_j$. Let $Spin \in R[G]$ denote the virtual representation given by the index of the spin$^c$-Dirac operator.

\begin{proposition}\label{prop:lef}
If $p$ is odd then:
\begin{align*}
\chi_{Spin}(g^2) &= \sum_i \frac{ \omega^{w_i} }{(\omega^{a_i} - \omega^{-a_i})(\omega^{b_i} - \omega^{-b_i})} \\
&+ \sum_j \left( \frac{1}{2}\langle c(\mathfrak{s}) , [\Sigma_j] \rangle \frac{ \omega^{v_j} }{(\omega^{c_j} - \omega^{-c_j})} -  \frac{1}{2}[\Sigma_j]^2 \frac{(\omega^{c_j} + \omega^{-c_j}) \omega^{v_j}}{(\omega^{c_j} - \omega^{-c_j})^2} \right).
\end{align*}
\end{proposition}
\begin{proof}
Since $g$ is assumed to have odd order $p$, the fixed point sets of $g$ and $g^2$ coincide. The Lefschetz index theorem \cite{as} expresses $\chi_{Spin}(g^2)$ as a sum of contributions from the components of the fixed point set of $g$. 

Consider first the contribution from an isolated point $x$ with normal weights $(a,b)$. A neighbourhood $U_x$ of this point can be identified with $\mathbb{C}^2$ with $g$ acting as $g(x,y) = ( \omega^a x , \omega^b y)$. The action of $g$ preserves a complex structure (in a neighbourhood of $x$). Let $S^{\pm}_{can}$ denote the canonical spin$^c$-structure associated to this complex structure, so $S^+_{can} = \mathbb{C} \oplus \mathbb{C}_{a+b}$ and $S^-_{can} = \mathbb{C}_a \oplus \mathbb{C}_b$. Let $S^{\pm}$ denote the spinor bundles associated to $\mathfrak{s}$. Then $S^{\pm}|_U \cong E \otimes S^{\pm}_{can}$ for some equivariant line bundle $E$. Then $E$ must have the form $E \cong \mathbb{C}_{u}$ for some $u \in \mathbb{Z}_p$. Taking determinants gives $L|_{U_x} = det(S^+|_{U_x}) \cong E^2 \otimes det(S^+_{can}) \cong \mathbb{C}_{2u+a+b}$. So if $w$ denotes the weight of $g$ acting on $L|_x$, then $w = 2u+a+b$.

The contribution to $\chi_{Spin}(g^2)$ coming from $x$ depends only on the normal weights $(a,b)$ and on the restriction to $x$ of the symbol of the spin$^c$-Dirac operator. Using $S^{\pm}|_{U_x} \cong E \otimes S^{\pm}_{can}$, the symbol agrees with that of the Dolbeault complex twisted by $E$. Hence the local contribution is:
\[
\frac{ \omega^{2u} }{(1-\omega^{-2a})(1-\omega^{-2b})} = \frac{ \omega^{2u+a+b} }{(\omega^a - \omega^{-a})(\omega^b - \omega^{-b})} = \frac{ \omega^w }{(\omega^a - \omega^{-a})(\omega^b - \omega^{-b})}.
\]

Consider now the contribution from a surface $\Sigma \subset X$ of genus $g$. Choose an orientation on $\Sigma$. This induces an orientation on the normal bundle $N$ such that $TX|_\Sigma = T\Sigma \oplus N$ as oriented bundles. Moreover $N$ is an oriented rank $2$ vector bundle on $\Sigma$ so we can regard it as a complex line bundle. The degree of $N$ is $[\Sigma]^2$. We have that $g$ acts on $N$ with some non-zero weight $c$. Since $T\Sigma$ and $N$ are complex line bundles, this defines a complex structure on $TX|_{\Sigma}$. Let $S^{\pm}_{can}$ denote the spinor bundles of the corresponding canonical spin$^c$-structure. Then $S^{\pm}|_{\Sigma} \cong E \otimes S^{\pm}_{can}$ for some complex line bundle $E$. Let $e$ denote the degree of $E$ and let $\alpha$ denote the weight of the action of $g$ on $E$. Then since $L|_\Sigma = det(S^{\pm})|_\Sigma = E^2 \otimes det(S^{\pm}_{can})$, we get
\begin{equation}\label{equ:degL}
deg(L|_\Sigma) = 2e + deg(N) + deg(T\Sigma) = 2e + [\Sigma]^2 + 2-2g,
\end{equation}
and
\[
v = 2\alpha + c
\]
where $v$ denotes the weight of $g$ acting on $L|_\Sigma$. Noting that $deg(L|_\Sigma) = \langle c(\mathfrak{s}) , [\Sigma] \rangle$, Equation (\ref{equ:degL}) can be re-written as
\[
e = g-1 - \frac{1}{2}[\Sigma]^2 + \frac{1}{2}\langle c(\mathfrak{s}) , [\Sigma] \rangle.
\]

Since the symbol of the spin$^c$-Dirac operator restricted to $\Sigma$ agrees with the symbol of the Dolbeault complex twisted by $E$, the local contribution to $\chi_{Spin}(g^2)$ coming from $\Sigma$ is:
\[
\int_\Sigma \frac{ Ch(E)(g^2) \cdot \mathcal{U} \cdot Td(\Sigma) }{1-\omega^{-2c} }
\]
where
\[
Ch(E)(g^2) = \omega^{2\alpha}(1 + c_1(E) ) = \omega^{v-c}(1+c_1(E))
\]
and
\[
\mathcal{U} = \left( \frac{1 - \omega^{-2c}e^{-c_1(N)} }{1 - \omega^{-2c} }\right)^{-1} = \left( 1 + \frac{ \omega^{-2c} c_1(N) }{1-\omega^{-2c}} \right)^{-1} = 1 - \frac{ \omega^{-2c} c_1(N) }{1-\omega^{-2c}}.
\]
So we find the contribution from $\Sigma$ is
\[
\frac{1}{2} \langle c(\mathfrak{s}) , [\Sigma] \rangle \frac{ \omega^v}{(\omega^c - \omega^{-c})} - \frac{1}{2}[\Sigma]^2 \frac{(\omega^c + \omega^{-c})}{(\omega^c - \omega^{-c})^2} \omega^v.
\]
Summing up all the contributions from points and surfaces gives the result.

\end{proof}

Let $d_u$ denote the multiplicity of $\mathbb{C}_u$ in $D$. Then
\[
Spin(g^2) = \sum_u d_u \omega^{2u} = \sum_u d_{2'u} \omega^u
\]
where $2'$ denote the multiplicative inverse of $2$ mod $p$.

Using same lifting and reducing procedure as in the case of the $G$-signature theorem, but applied to the spin $^c$-Lefschetz index formula, we have:
\begin{align}\label{equ:gspin1}
(t-1)^2 \sum_{u=0}^{p-1} d_{2'u} t^u &=\sum_i t^{w_i} t^{a_i+b_i} \psi_{2a_i}(t) \psi_{2b_i}(t) \\
& + \sum_j \frac{1}{2}\langle c(\mathfrak{s}) , [\Sigma_j] \rangle t^{v_j}(t-1) t^{c_j} \psi_{2c_j}(t) \nonumber \\
&  - \sum_j \frac{1}{2}[\Sigma_j]^2 t^{v_j} (t^{3c_j} + t^{c_j}) \psi_{2c_j}(t)^2 \text{ in } \mathbb{Z}_p[x]/(x^5). \nonumber 
\end{align}

As before, we will expand both sides in terms of $x$ and equate coefficients. We have
\begin{align*}
t^{a}\psi_{2a}(t) &= \frac{1}{2a} + \frac{1}{4a}x - \frac{(2a^2+1)}{24a}x^2 + \frac{(2a^2+1)}{48a}x^3 \\
& \quad \quad + \frac{(14a^4 - 40a^2-19)}{1440a}x^4 \text{ in } \mathbb{Z}_p[x]/(x^5).
\end{align*}

Multiplying together two such expansions we get
\begin{align*}
t^{a+b}\psi_{2a}(t)\psi_{2b}(t) &= \frac{1}{4ab} + \frac{1}{4ab}x + \frac{1}{48}\left(\frac{-2a^2-2b^2+1}{ab}\right)x^2 \\
& \quad \quad +\frac{1}{2880}\left( \frac{ 14(a^4 + b^4) +20 a^2b^2 - 3}{ab} \right)x^4 \text{ in } \mathbb{Z}_p[x]/(x^5).
\end{align*}

We also have
\begin{align*}
(t^{3c} + t^c)\psi_{2c}^2 &= \frac{1}{2c^2} + \frac{1}{2c^2}x + \frac{(2c^2+1)}{24c^2}x^2 \\
& \quad \quad - \frac{(14c^4+1)}{480c^2}x^4 \text{ in } \mathbb{Z}_p[x]/(x^5).
\end{align*}

For the Lefschetz index theorem, we will need to expand $x^2 \sum_u d_{2'u} t^{u}$ up to order $x^4$:
\[
x^2 \sum_u d_{2'u} t^u = x^2 d +  \sum_u d_{2'u}\left( ux^3 + \binom{u}{2}x^4 \right) \text{ in } \mathbb{Z}_p[x]/(x^5)
\]
where $d = \sum_{u=0}^{p-1} d_u = (c(\mathfrak{s})^2 - \sigma(X))/8$ is the index of the Dirac operator.

Equating coefficients, we obtain the following identities:
\begin{equation}\label{equ:dirx1}
\sum_i  \frac{w_i}{a_i b_i} + \sum_j  \frac{\langle c(\mathfrak{s}) , [\Sigma_j] \rangle}{c_j} - \sum_j \frac{[\Sigma_j]^2 v_j}{c_j^2} = 0
\end{equation}

valid for all odd $p$.

\begin{equation}\label{equ:dirx2}
c(\mathfrak{s})^2 = \sum_i \frac{w_i^2}{a_i b_i} + 2\sum_j \frac{\langle c(\mathfrak{s}) , [\Sigma_j] \rangle v_j}{c_j} - \sum_j \frac{[\Sigma_j]^2 v_j^2}{c_j^2}
\end{equation}

valid for all $p \ge 5$. In fact, we will show in Section \ref{sec:loc} that this also holds when $p=3$.

\begin{align}
24 \sum_u d_{2'u} u &= \sum_i \frac{( w_i^3 - w_i(a_i^2 + b_i^2))}{a_ib_i} + \sum_j \langle c(\mathfrak{s}) , [\Sigma_j] \rangle \frac{( 3v_j^2 - c_j^2)}{c_j} \label{equ:dirx3} \\
& - \sum_j [\Sigma_j]^2 \frac{( v_j^3 + v_j c_j^2)}{c_j^2} \nonumber
\end{align}
valid for all $p \ge 5$.

\begin{align}
720 \sum_u d_{2'u} u^2 &= \sum_i \frac{( 15 w_i^4 - 30 w_i^2(a_i^2 + b_i^2) - 4a_i^2 b_i^2 + 7(a_i^2 + b_i^2)^2 )}{a_ib_i} \label{equ:dirx4} \\
&+ 60\sum_j \langle c(\mathfrak{s}) , [\Sigma_j] \rangle \frac{( v_j^3  - v_j c_j^2 )}{c_j} \nonumber \\
&+ \sum_j [\Sigma_j]^2 \frac{( -15 v_j^4 - 30v_j^2 c_j^2 + 21 c_j^4 )}{c_j^2} \nonumber
\end{align}

valid for all $p \ge 7$.

\section{Localisation}\label{sec:loc}

Let $G = \mathbb{Z}_p = \langle g \rangle$ ($p$ an odd prime) act smoothly on a closed $4$-manifold $X$. Let $x$ denote a generator of $H^2_{\mathbb{Z}_p}(pt ; \mathbb{Z})$, then $H^*_{\mathbb{Z}_p}(pt ; \mathbb{Z}) \cong \mathbb{Z}[x]/( px )$. The localised cohomology ring $x^{-1}H^*_{\mathbb{Z}_p}(pt ; \mathbb{Z})$ is isomorphic to $\mathbb{Z}_p[x,x^{-1}]$. 

Let $H^*_{loc}(X) = x^{-1} H^{ev}_{\mathbb{Z}_p}(X ; \mathbb{Z})$ be the even part of the equivariant cohomology ring of $X$ localised with respect to $x$. Then $H^*_{loc}(X) \cong H^0_{loc}(X)[x,x^{-1}]$.

Let $F$ denote the fixed point set of $g$. Then $F$ consists of $n$ isolated points $x_1, \dots , x_n$ and $m$ surfaces $\Sigma_1, \dots , \Sigma_m$. Since $p$ is odd, the surfaces $\Sigma_j$ are all orientable and we fix orientations of each component. Let $\iota : F \to X$ be the inclusion map. The localisation theorem in equivariant cohomology \cite[III (3.8)]{die} implies that $\iota^* : H^{ev}_{loc}(X) \to H^{ev}_{loc}(F)$ is an isomorphism. Since both rings are algebras over $\mathbb{Z}_p[x,x^{-1}]$, we will mainly consider the degree $0$ part $\iota^* : H^0_{loc}(X) \to H^0_{loc}(F)$.

For an isolated point $x_i$, we have $H^0_{loc}(x_i) \cong \mathbb{Z}_p$. For a surface $\Sigma_j$, we have $H^0_{loc}(\Sigma_j) \cong \mathbb{Z}_p[ \tau ]/(\tau^2)$. Here $\tau = x^{-1} \nu$, where $\nu$ is the generator of $H^2_{\mathbb{Z}_p}(\Sigma_j ; \mathbb{Z}) \cong \mathbb{Z}$ corresponding to the chosen orientation. Hence we have a ring isomorphism
\[
H^0_{loc}(X) \cong \bigoplus_i \mathbb{Z}_p \oplus \bigoplus_j \frac{\mathbb{Z}_p[\tau]}{(\tau^2)}.
\]

Since $X$ and $F$ are oriented, various push-forward maps are defined. Namely, let $p : X \to pt$, $\pi : F \to pt$ be the unique maps to a point. These maps are equivariant and $\pi = p \circ \iota$. So we get pushforward maps
\begin{align*}
p_* &: H^{ev}_{\mathbb{Z}_p}(X ; \mathbb{Z}) \to H^{ev}_{\mathbb{Z}_p}(pt ; \mathbb{Z}), \\
\pi_* &: H^{ev}_{\mathbb{Z}_p}(F ; \mathbb{Z}) \to H^{ev}_{\mathbb{Z}_p}(pt ; \mathbb{Z}), \\
\iota_* &: H^{ev}_{\mathbb{Z}_p}(F ; \mathbb{Z}) \to H^{ev}_{\mathbb{Z}_p}(X ; \mathbb{Z}),
\end{align*}

satisfying $\pi_* = p_* \circ \iota_*$. These maps commute with $x$ and so induce corresponding maps on $H^{ev}_{loc}$. The maps change degrees, however we would like to work throughout with just $H^0_{loc}$. To do this we define adjusted maps
\begin{align*}
P &: H^0_{loc}(X) \to H^0_{loc}(pt) = \mathbb{Z}_p, \\
\Pi &: H^0_{loc}(F) \to H^0_{loc}(pt) = \mathbb{Z}_p, \\
I &: H^0_{loc}(F) \to H^0_{loc}(X)
\end{align*}

as follows. For $P$ we simply take $P = x^2 p_*$. For $\Pi$ and $I$ we first write $F = F_0 \cup F_2$, where $F_0$ are the isolated points and $F_2$ are the surfaces. Then $\pi_* = (\pi_0)_* + (\pi_2)_*$, where $\pi_0 : F_0 \to pt$, $\pi_2 : F_2 \to pt$ are the unique maps to $pt$. We then set $\Pi = \Pi_0 + \Pi_2$, where $\Pi_0 = (\pi_0)_*$ and $\Pi_2 = x(\pi_2)_*$. Similarly, set $\iota_0 = \iota|_{F_0}$, $\iota_2 = \iota|_{F_2}$ and then set $I = x^{-2} (\iota_0)_* + x^{-1} (\iota_2)_*$. Clearly $\Pi_0 = P \circ I_0$ and $\Pi_2 = P \circ I_2$, so $\Pi = P \circ I$.

Let $N$ denote the normal bundle of $F$ (so $N$ has rank $4$ over $F_0$ and rank $2$ over $F_2$). Then $\iota^*( \iota_*( u ) ) = e(N) u $ for any $u \in H^{ev}_{loc}(F)$, where $e(N)$ is the Euler class of $N$. Now if $x_i$ is an isolated point with normal weights $(a_i,b_i)$, then $e(N)|_{x_i} = a_i^2 b_i^2 x^4$. If $\Sigma_j$ is a fixed surface with self-intersection $\delta_j = [\Sigma_j]^2$ and normal weight $c_j$, then $e(N)|_{\Sigma_j} = c_jx + \delta_j \nu$. The Euler class of $N$ is invertible in $H^{ev}_{loc}(F)$. In fact, we have
\[
e(N)|_{x_i}^{-1} = a_i^{-1} b_i^{-1} x^{-2}, \quad e(N)|_{\Sigma_j}^{-1} = c_j^{-1} x^{-1} - \delta_j c_j^{-2} x^{-2} \nu.
\]

Let $\omega \in H^{ev}_{loc}(X)$. Then by the localisation theorem $\omega = \iota_*(\omega_F)$ for some $\omega_F \in H^{ev}_{loc}(F)$. Then $\iota^*(\omega) = \iota^*( \iota_*(\omega_F)) = e(N) \omega_F$. Since $e(N)$ is invertible, this gives $\omega_F = e(N)^{-1} \iota^*(\omega)$. Therefore,
\[
p_*(\omega) = p_*( \iota_*(\omega_F) ) = \pi_*(\omega_F) = \pi_*( e(N)^{-1} \iota^*(\omega) ).
\]
Set $N_i = N|_{F_i}$ for $i=0,2$. Then we can write
\[
p_*(\omega) = (\pi_0)_*( e(N_0)^{-1} \iota_0^*(\omega) ) + (\pi_2)_*( e(N_2)^{-1} \iota_2^*(\omega)).
\]

Define $\alpha_0 = x^2 e(N_0)^{-1}$, $\alpha_2 = x e(N_2)^{-1}$. So $\alpha = \alpha_0 + \alpha_2 \in H^0_{loc}(F)$ and it follows that
\[
P(\omega) = \Pi_0( \alpha_0 \omega_0) + \Pi_2( \alpha_2 \omega_2),
\]
where $\omega_i = \iota^*_i(\omega)$ for $i=0,2$. Note that
\[
\alpha|_{x_i} = a_i^{-1} b_i^{-1}, \quad \alpha|_{\Sigma_j} = c_j^{-1} - \delta_j c_j^{-2} \tau.
\]

Now $\Pi_0( w_1 , \dots , w_n) = w_1 + \cdots + w_n$ and $\Pi_2( (v_1 + u_1 \tau) , \dots , (v_m + u_m \tau) ) = u_1 + \cdots + u_m$. So if $\omega_0 = (\omega_1 , \dots , \omega_n)$, $\omega_2 = ( (v_1 + u_1 \tau) , \dots , (v_m + u_m \tau))$, then
\[
P(\omega) = \sum_i \frac{ \omega_i }{a_ib_i} - \sum_j [\Sigma_j]^2 \frac{v_j}{c_j^2} + \sum_j \frac{u_j}{c_j}.
\]

Now for example, if we take $\omega = 1$, then since $p_*(1) = 0$ (for degree reasons), we get that
\[
0 = P(1) = \sum_i \frac{1}{a_ib_i} - \sum_j [\Sigma_j]^2 \frac{1}{c_j^2}.
\]

We recognise this as Equation (\ref{equ:gsigp1}), which we obtained in Section \ref{sec:index} from the $G$-signature theorem. Here we have given an alternative proof. 

Next, let $L \to X$ be an equivariant line bundle and set $\alpha = c_1(L) \in H^2_{\mathbb{Z}_p}(X ; \mathbb{Z})$. Then $p_*(\alpha) = 0$ for degree reasons. On the other hand, if we define $w_i,v_i,u_i \in \mathbb{Z}_p$ by $\alpha|_{x_i} = w_i x$, $\alpha|_{\Sigma_j} = v_j x + u_j\tau$, then
\begin{equation}\label{equ:wtlineb}
0 = P(x^{-1}\alpha) = \sum_i \frac{ w_i }{a_i b_i } - \sum_j [\Sigma_j]^2 \frac{v_j}{c_j^2} + \sum_j u_j \frac{1}{c_j}.
\end{equation}

Moreover, we observe that $w_i,v_i$ are the weights of $L|_{x_i}$, $L|_{\Sigma_j}$ and $u_j = \langle [L] , [\Sigma_j] \rangle$.

We recognise this as one of the equations obtained from the spin$^c$ Lefschetz index theorem (in the special case that $\alpha$ is characteristic). Here we have proven it for all equivariant line bundles. If we take $L$ to be the determinant line of an equivariant spin$^c$-structure $\mathfrak{s}$, then Equation (\ref{equ:wtlineb}) gives
\[
0 = \sum_i \frac{w_i}{a_ib_i} - \sum_j [\Sigma_j]^2 \frac{v_j}{c_j^2} + \sum_j \frac{ \langle c(\mathfrak{s}) , [\Sigma_j] \rangle}{c_j},
\]
which is precisely Equation (\ref{equ:dirx1}).

We remark here that a different proof of Equation (\ref{equ:wtlineb}) in the case of homologically trivial actions was given in \cite{ah} using the twisted signature operator.

Let $L_1,L_2$ be two equivariant line bundles. Set $\alpha_i = c_1(L_1) \in H^2_{\mathbb{Z}_p}(X ; \mathbb{Z})$ for $i=1,2$. Then $\alpha_1 \alpha_2$ has degree $4$, so $P(\alpha_1 \alpha_2/x^2) = p_*(\alpha_1 \alpha_2 ) = \langle [L_1] , [L_2] \rangle$. Hence we have
\[
\langle [L_1] , [L_2] \rangle = \sum_i \frac{ w_{1,i} w_{2,i} }{a_i b_i} - \sum_j [\Sigma_j]^2 \frac{v_{1j} v_{2j} }{c_j^2} + \sum_j \frac{v_{1j}u_{2j} + v_{2j}u_{1j}}{c_j},
\]
where $w_{a,i},v_{a,j}$ are the weights of $L_a$, and $u_{aj} = \langle [L_a] , [\Sigma_j] \rangle$. Specialising to the case where $L_1 = L_2$ are the determinant line bundle of an equivariant spin$^c$-structure $\mathfrak{s}$, we obtain
\[
c(\mathfrak{s})^2 = \sum_i \frac{ w_{i}^2 }{a_i b_i} - \sum_j [\Sigma_j]^2 \frac{v_{j}^2}{c_j^2} + 2\sum_j \frac{\langle c(\mathfrak{s}) , [\Sigma_j] \rangle v_j}{c_j},
\]
which is precisely Equation (\ref{equ:dirx2}). This shows that Equation (\ref{equ:dirx2}) holds even when $p=3$.

We use $\iota^*$ to identify $H^0_{loc}(F)$ with $H^0_{loc}(X)$ and we use $H$ to denote either of them. We also define $\langle \; \rangle : H \to \mathbb{Z}_p$ to be $\langle f \rangle = P(f)$. We define a bilinear form $\langle \, , \, \rangle $ on $H$ by setting $\langle f , g \rangle = \langle fg \rangle$. Since $\langle f , g \rangle = P(fg) = x^2 p_*(fg)$, the bilinear form is non-degenerate as a consequence of Poincar\'e duality. Equipped with this structure, $H$ is a Frobenius ring.

Let $p_1 = p_1(X)$ be the first Pontryagin class of $X$. Because $G$ acts on $TX$, $p_1$ lifts canonically to equivariant cohomology. Set $\rho = x^{-2}p_1$ so that $\rho \in H$. At $x_i$ the tangent bundle of $X$ is $\mathbb{C}_{a_i} \oplus \mathbb{C}_{b_i}$, so $p_1|_{x_i} = (a_ix)^2 + (b_ix)^2$. On $\Sigma_j$, we have $TX|_{\Sigma_j} = N_{\Sigma_j} \oplus T\Sigma_j$, where $N_{\Sigma_j}$ is the normal bundle to $\Sigma_j$. Hence 
\[
p_1|_{\Sigma_j} = c_1(N_{\Sigma_j})^2 + c_1(T\Sigma_j)^2 = c_1(N_{\Sigma_j})^2 = (c_j x + \delta_j \nu)^2 = x^2( c_j^2 + 2c_j \delta_j \tau).
\]
So we have
\[
\rho|_{x_i} = a_i^2 + b_i^2, \quad \rho|_{\Sigma_j} = c_j^2 + 2c_j \delta_j \tau.
\]

Now the (ordinary) signature theorem on $X$ says $p_*(p_1) = 3\sigma(X)$. Thus 
\[
\langle \rho \rangle = 3\sigma(X).
\]
Explicitly in terms of $a_i,b_i,c_j, [\Sigma_j]^2$, this becomes
\[
\sum_i \frac{ a_i^2 + b_i^2}{a_ib_i} + \sum_j [\Sigma_j]^2 = 3\sigma(X)
\]
which we recognise as Equation (\ref{equ:gsigp2}). This shows that Equation (\ref{equ:gsigp2}) is valid even for $p=3$.

Let $e(X)$ denote the Euler class of $X$ and set $e = x^{-2}e(X) \in H$. Then 
\[
e|_{x_i} = a_i b_i, \quad e|_{\Sigma_j} = (2-2g_j)c_j \tau.
\]
Then for example $\langle e  \rangle = \chi(X)$, which reads as
\[
\sum_i 1 + \sum_j (2-2g_j) = \chi(X),
\]
which is the Lefschetz fixed point theorem.

From $e$ and $\rho$ we can form tautological classes
\[
\kappa_{a,b} = \langle \rho^a e^b \rangle \in \mathbb{Z}_p.
\]
From the expressions for $e$ and $\rho$, we can easily calculate $\kappa_{a,b}$:
\begin{align*}
\langle \rho^a \rangle &= \sum_i \frac{(a_i^2 + b_i^2)^a}{a_ib_i} + (2a-1)\sum_j [\Sigma_j]^2 c_j^{2a-2} \quad a \ge 1, \\
\langle e^b \rangle &= \sum_i (a_i b_i)^{b-1} \quad b \ge 2, \\
\langle \rho^a e \rangle &= \sum_i (a_i^2 + b_i^2)^a + \sum_j (2-2g_j) c_j^{2a} \quad a \ge 1, \\
\langle \rho^a e^b \rangle &= \sum_i (a_i^2 + b_i^2)^a (a_i b_i)^{b-1} \quad a \ge 1, b \ge 2.
\end{align*}

Equation (\ref{equ:gsigp2}) can be re-written as:
\[
\langle \rho^2 \rangle - 7 \langle e^2 \rangle = -90 \sum_u u^2 \sigma_u
\]
which holds for all $p \ge 7$.

Let $\mathfrak{s}$ be an equivariant spin$^c$-structure and set $w = c(\mathfrak{s})/x \in H$. Equations (\ref{equ:dirx3}) and (\ref{equ:dirx4}) can be rewritten as:
\begin{align}
\langle w^3 \rangle - \langle w \rho \rangle &= 24 \sum_u d_{2'u} u, \quad (p \ge 5).  \label{equ:dirx32} \\
15\langle w^4 \rangle - 30\langle w^2 \rho \rangle + 7\langle \rho^2 \rangle - 4 \langle e^2 \rangle &= 720 \sum_u d_{2'u} u^2, \quad (p \ge 7). \label{equ:dirx42}
\end{align}

\section{Cyclic group actions on definite $4$-manifolds}\label{sec:definite}

In this section we specialise to the case that $X$ is a positive definite $4$-manifold. We assume that $H_1(X ; \mathbb{Z}_2) = 0$. The intersection form on $X$ is diagonal by Donaldson's theorem. Let $G = \mathbb{Z}_p = \langle g \rangle$ act smoothly on $X$ where $p$ is an odd prime. Let $\Lambda = H^2(X ; \mathbb{Z})$ be the intersection form of $X$. We will show that the action of $G$ on $\Lambda$ orthogonally decomposes into a sum of trivial and regular summands. Let $T$ denote the rank $1$ lattice with generator $e$ such that $\langle e , e \rangle = 1$ and let $G$ act trivially. Let $R$ denote the lattice with basis $e_0, e_1 , \dots , e_{p-1}$ with $\langle e_i , e_j \rangle = \delta_{ij}$ and let $G$ act on $R$ by $ge_i = e_{i+1}$ (where the index $i$ is considered modulo $p$).

\begin{proposition}\label{prop:decomp}
There exists $t,r \ge 0$ such that $\Lambda \cong T^t \oplus R^r$ as $G$-modules with bilinear forms.
\end{proposition}
\begin{proof}
Equip $\mathbb{Z}^b$ with the standard bilinear form. We use induction on $b$ to prove that any orthogonal action of $G$ is isomorphic to an orthogonal direct sum of copies of $T$ and $R$. The result is clearly true if $b = 0$. Now suppose $b > 0$ and the result is true for ranks less than $b$. Let $e \in \mathbb{Z}^b$ be a class with $e^2 = 1$. Consider the orbit $Ge = \{ g^je \; | j = 0,1,\dots , p-1\}$. Since $p$ is prime $|Ge| = 1$ or $p$. If $|Ge| = 1$, then $ge = e$ and $\mathbb{Z}^b$ decomposes equivariantly as $\mathbb{Z}^b = \Lambda_1 \oplus \Lambda_2$, with $\Lambda_1 = \mathbb{Z}e \cong T$ and $\Lambda_2 = e^{\perp}$. Then by induction, $\mathbb{Z}^b \cong T^t \oplus R^r$ for some $t,r$.

If $|Ge| = p$ then $e_0,e_1, \dots , e_{p-1}$ where $e_j = g^j p$ are disctinct classes of norm $1$. If $e_i = -e_j$ for some $i,j$ then $g^i e = -g^j e$ and hence $e = (g^{i})^p e = (-1)^p (g^{j})^p e = -e$, which is impossible. So no two $e_i,e_j$ are proportional to one another which implies that $\langle e_i , e_j \rangle = 0$ for $i \neq j$. Thus $\mathbb{Z}^b$ decomposes equivariantly as $\mathbb{Z}^b = \Lambda_1 \oplus \Lambda_2$ with $\Lambda_1 \cong R$ the span of $e_0, \dots , e_{p-1}$ and $\Lambda_2$ their orthogonal complement. Hence by induction $\mathbb{Z}^b \cong T^t \oplus R^r$ for some $t,r$.
\end{proof}

Let $t,r$ denote the number of summands of $T$ and $R$ in $\Lambda = H^2(X ; \mathbb{Z})$. So $b_2(X) = \sigma(X) = t + pr$. Consider the Borel spectral sequence $E_r^{p,q}$ for the equivariant cohomology of $X$. We have $E_2^{p,q} = H^p( \mathbb{Z}_p ; H^q(X ; \mathbb{Z}) )$. This is concentrated in even degrees, so there are no differentials, $H^0_{loc}(X) \cong \mathbb{Z}_p^{t+2}$ (as abelian groups) and $H^1_{loc}(X) = 0$. Let $F = X^G$ denote the fixed point set. Then $F$ consists of $n$ isolated points $x_1, \dots , x_n$ and $m$ orientable surfaces $\Sigma_1, \dots , \Sigma_m$. From the localisation theorem, it follows that every surface component $\Sigma_j$ has genus zero and $n+2m = t+2$.

Since $H^2(X ; \mathbb{Z}) \cong T^t \oplus R^r$, we may choose an orthonormal basis $e_i , f_{j,k} \in H^2(X ; \mathbb{Z})$ where $1 \le i \le t$, $1 \le j \le r$, $0 \le k \le p-1$ such that the action of $g$ is given by $ge_i = e_i$ for all $i$ and $gf_{j,k} = f_{j, k+1}$ for all $j,k$. Then a basis for $H^2(X ; \mathbb{Z})^G$ is given by $e_i, f_j$, $1 \le i \le t$, $1 \le j \le r$ where we set $f_j = f_{j,0} + f_{j,1} + \cdots + f_{j,p-1}$. Notice that $\langle e_i , e_j \rangle = \delta_{ij}$, $\langle e_i , f_j \rangle = 0$, $\langle f_i , f_j \rangle = p \delta_{ij}$. Set $f = f_1 + \cdots + f_r$. 

From the Borel spectral sequence we obtain a short exact sequence
\[
0 \to H_{\mathbb{Z}_p}^2(pt ; \mathbb{Z}) \to H^2_{\mathbb{Z}_p}(X ; \mathbb{Z}) \to H^2(X ; \mathbb{Z})^G \to 0.
\]
Thus every class in $H^2(X ; \mathbb{Z})^G$ has a lift to $H^2_{\mathbb{Z}_p}(X ; \mathbb{Z})$ and the lift is unique up to addition of a multiple of $x$, the generator of $H^2_{\mathbb{Z}_p}(pt ; \mathbb{Z})$.

Given a class $a \in H^2(X ; \mathbb{Z})^G$, let $\widehat{a} \in H^2_{\mathbb{Z}_p}(X ; \mathbb{Z})$ be a lift to equivariant cohomology. Then $\widehat{a}$ corresponds to an equivariant line bundle. We would like to determine the weights $w_i, v_j$ of this line bundle in terms of $\widehat{a}$. We have that $H^2(X ; \mathbb{Z})^G \cong \mathbb{Z}^{t+r}$ with basis $e_1, \dots , e_t, f_1 , \dots , f_r$. We will show that the problem of determining weights for lifts of the $f_i$ can easily be solved. In contrast determining the weights for lifts of the $e_i$ is highly non-trivial.

Let $F_{j,0}$ denote a line bundle with $c_1(F_{j,0}) = f_{j,0}$. Then $f_j = c_1(F_j)$, where
\[
F_j \cong F_{j,0} \otimes f^*(F_{j,0}) \otimes \cdots \otimes (f^{p-1})^*(F_{j,0}).
\]
At a point $z \in X$, the fibre of $F_j$ is given by
\[
(F_j)_z = (F_{j,0})_{z} \otimes (F_{j,0})_{gz} \otimes \cdots \otimes (F_{j,0})_{g^{p-1}z}.
\]
Consequently, there is a canonical way to make $F_j$ into an equivariant line bundle, namely $g : (F_j)_p \to (F_j)_{gz}$ is given by
\[
g( a_0 \otimes a_1 \otimes \cdots \otimes a_{p-1} ) = a_1 \otimes a_2 \otimes \cdots \otimes a_0
\]
where $a_i \in (F_{j,0})_{g^i z}$. With this definition, it is easily seen that the weights of $F_j$ over each fixed point are zero (if $gz = z$, then we can choose $a_0 = a_1 = \cdots = a_{p-1}$ to get $g(a_0 \otimes \cdots \otimes a_{p-1}) = a_0 \otimes \cdots \otimes a_{p-1}$). Thus the problem of determining the weights for the $f_j$ is completely solved.

Now we turn to the problem of detemining the weights for lifts of the $e_a$. Choose (arbitrarily) lifts $\widehat{e}_a$ of $e_a$ to equivariant cohomology. We also take lifts $\widehat{f}_j$ of $f_j$ to equivariant cohomology using the $G$-action on $F_j$ defined above. We also set $\widehat{f} = \widehat{f}_1 + \cdots + \widehat{f}_r$. Given $\epsilon = (\epsilon_1 , \dots , \epsilon_t) \in \{ 1,-1\}^t$, we set 
\[
c(\epsilon) = \epsilon_1 e_1 + \epsilon_2 e_2 + \cdots + \epsilon_t e_t + f.
\]
Then $c(\epsilon)$ is a characteristic and $c(\epsilon)^2 = b_2(X) = \sigma(X)$. Corresponding to $c(\epsilon)$ is a spin$^c$-structure $\mathfrak{s}(\epsilon)$ with $c( \mathfrak{s}(\epsilon) ) = c(\epsilon)$. By Proposition \ref{prop:eqspinc} we can make $\mathfrak{s}$ into an equivariant spin$^c$-structure whose determinant line has equivariant first Chern class given by $\epsilon_1 \widehat{e}_1 + \cdots + \epsilon_t \widehat{e}_t + \widehat{f}$.

For $1 \le a \le t$, we let $w_a = x^{-1}e_a \in H$ and $z = x^{-1}f \in H$. 

\begin{lemma}
We have $z = 0$.
\end{lemma}
\begin{proof}
Indeed, since $f_1, \dots , f_r$ come from the regular representation summands in $H^2(X ; \mathbb{Z})$, we have $xf_i = 0$ for all $i$ and $xf = 0$ as well. Thus on passing to $H^*_{loc}(X)$ we have that $f_i = 0$ for all $i$ and $f=0$.
\end{proof}

By the lemma, this means that $x^{-1}c(\epsilon) = \epsilon_1 w_1 + \cdots + \epsilon_t w_t$ in $H$.

Since $w_a = x^{-1}\widehat{e}_a$, where $\widehat{e}_a$ is a degree $2$ equivariant cohomology class, we have
\[
\langle w_a \rangle = 0 \text{ for all } a.
\]

Since $p_*(\widehat{e}_a \widehat{e}_b) = \langle e_a \smallsmile e_b , [X] \rangle = \delta_{ab}$, it follows that
\[
\langle w_a w_b \rangle = \delta_{ab}.
\]

\begin{proposition}
For any $\epsilon = (\epsilon_1, \dots , \epsilon_t) \in \{ 1,-1\}^t$ we have:
\begin{align}
\langle w(\epsilon)^3 \rangle - \langle \rho w(\epsilon) \rangle &= 0, \quad (p \ge 5). \label{equ:dirx33} \\
15\langle w(\epsilon)^4 \rangle - 30\langle w(\epsilon)^2 \rho \rangle + 7\langle \rho^2 \rangle - 4 \langle e^2 \rangle &= 0, \quad (p \ge 7). \label{equ:dirx43}
\end{align}
Furthermore, for any $\epsilon \in \{ 1,-1\}^t$ and any $a \in \{1,\dots , t\}$ we have
\begin{equation}
60\langle w^4 \rangle - 120 \langle w^2 \rho \rangle + 28\langle \rho^2 \rangle - 16\langle e^2 \rangle = 5\left( \langle w^3 \rangle - \langle w\rho \rangle \right)^2, \quad (p \ge 7). \label{equ:dirx44}
\end{equation}
where $w = 2w_a + w(\epsilon)$.
\end{proposition}
\begin{proof}
Since $X$ is positive definite, an application of equivariant Seiberg--Witten theory \cite[Theorem 9.1]{bar} (applied to $X$ with the opposite orientation so that it is negative definite) gives that the index $Spin \in R[G]$ of $\mathfrak{s}(\epsilon)$ satisfies $d_u \ge 0$. But since $c(\epsilon)^2 = \sigma(X)$, we have that $d = \sum_u d_u = 0$. Hence $d_u = 0$ for all $u$. Thus Equations (\ref{equ:dirx32}), (\ref{equ:dirx42}), reduce to (\ref{equ:dirx33}), (\ref{equ:dirx43}).

Now let $w = 2w_a + w(\epsilon)$. If $\epsilon_a = -1$ then $w$ is of the form $w = w(\epsilon')$ where $\epsilon'_a = 1$, $\epsilon'_b = \epsilon_b$ for $b \neq a$. In this case (\ref{equ:dirx43}) follows from (\ref{equ:dirx33}) and (\ref{equ:dirx43}). Now assume that $\epsilon_a = 1$. Then there is an equivariant spin$^c$-structure $\mathfrak{s}$ with $c(\mathfrak{s}) = 2\widehat{e_a} + \sum_b \epsilon_b \widehat{e}_b$ such that $w = x^{-1} c(\mathfrak{s}) \in H^*_{loc}(X)$. We have that $c(\mathfrak{s})^2 = \sigma(X) + 8$. Then the corresponding index $Spin \in R[G]$ is a virtual representation of rank $1$, that is, the multiplicities $d_u$ satisfy $\sum_u d_u = 1$. Another application of \cite[Theorem 9.1]{bar} gives $d_u \ge 0$ for all $u$. Therefore there exists a $v \in \mathbb{Z}_p$ such that $d_v = 1$ and $d_u = 0$ for $u \neq v$. Therefore, we have
\[
\sum_u d_{2'u} u^2 = \left( \sum_u d_{2'u} u \right)^2.
\]
Applying Equations (\ref{equ:dirx32}), (\ref{equ:dirx33}), we obtain (\ref{equ:dirx44}).
\end{proof}

Since $b_2(X) = t + pr = t \; ({\rm mod} \; p)$, the signature theorem gives
\begin{equation}\label{equ:sig3t}
\langle \rho \rangle = 3t.
\end{equation}

The identites (\ref{equ:dirx33}), (\ref{equ:dirx43}), (\ref{equ:dirx44}) together with (\ref{equ:sig3t}) motivates us to introduce the following definitions. Let $n,m$ be non-negative integers and let $H_{n,m}$ be the ring
\[
H_{n,m} = \mathbb{Z}_p^n \oplus \left( \frac{\mathbb{Z}_p[\tau]}{(\tau^2)} \right)^m.
\]
Elements of $H_{n,m}$ will be written as $h = ( x_1 , \dots , x_n ; y_1 , \dots , y_m )$, where $x_a \in \mathbb{Z}_p$, $y_j \in \mathbb{Z}_p[\tau]/(\tau^2)$. Sometimes we will use a shorthand notation where we write $h = ( \{ x_i \} ; \{ y_j \} )$.

Let $a_i,b_i,c_j,\delta_j \in \mathbb{Z}_p$, $1 \le i \le n$, $1 \le j \le m$, $a_i,b_i,c_i \neq 0$. Define $\langle \; \; \rangle : H_{n,m} \to \mathbb{Z}_p$ by:
\begin{equation}\label{equ:pairing}
\langle w \rangle = \sum_i \frac{w_i}{a_ib_i} - \sum_j \frac{\delta_j v_j}{c_j^2} + \sum_j \frac{u_j}{c_j}
\end{equation}
where $w = (w_1, \dots , w_n ; (v_1 + u_1\tau), \dots , (v_m + u_m \tau))$. We also define $\rho, e \in H_{n,m}$ by
\begin{equation}\label{equ:p1}
\rho = ( a_1^2 + b_1^2 , \dots , a_n^2 + b_n^2 ; ( c_1^2 + 2c_1\delta_1 \tau) , \dots , (c_m^2 + 2c_m \delta_m \tau) )
\end{equation}
and
\begin{equation}\label{equ:euler}
e = (a_1b_1 , \dots , a_n b_n ; 2c_1 \tau , \dots , 2c_m \tau).
\end{equation}

\begin{definition}\label{def:wtsys}
Let $n,m$ be non-negative integers such that $n+2m \ge 3$. A {\em weight system} of type $(n,m)$ is a collection of elements
\[
a,b \in \mathbb{Z}_p^n, c,\delta \in \mathbb{Z}_p^m, w_1, \dots , w_t \in H_{n,m}
\]
where $t = n+2m-2$ such that the following conditions hold.
\begin{itemize}
\item[(W1)]{$a_i,b_i,c_j \neq 0$ for all $i,j$.}
\item[(W2)]{$\langle 1 \rangle = 0$, where $1$ denotes the identity element of $H_{n,m}$.}
\item[(W3)]{$\langle w_a \rangle = 0$ for all $a$.}
\item[(W4)]{$\langle w_a w_b \rangle = \delta_{ab}$ for all $a,b$ ($\delta_{ab} = 1$ if $a=b$, $\delta_{ab}=0$ otherwise).}
\item[(W5)]{$\langle \rho \rangle = 3t$.}
\item[(W6)]{$\langle w(\epsilon)^3 \rangle - \langle \rho w(\epsilon) \rangle = 0$ for all $\epsilon$.}
\item[(W7)]{$15\langle w(\epsilon)^4 \rangle - 30\langle w(\epsilon)^2 \rho \rangle + 7\langle \rho^2 \rangle - 4 \langle e^2 \rangle = 0$ for all $\epsilon$.}
\item[(W8)]{For any $a,\epsilon$, if $w = 2w_a + w(\epsilon)$, then
\[
60\langle w^4 \rangle - 120\langle w^2 \rho \rangle + 28\langle \rho^2 \rangle - 16 \langle e^2 \rangle = 5( \langle w^3 \rangle - \langle \rho w \rangle )^2.
\]
}
\end{itemize}

In the above list, $\langle \; \rangle$ is given by Equation (\ref{equ:pairing}), $\rho, e \in H_{n,m}$ are given by Equations (\ref{equ:p1}), (\ref{equ:euler}) and $w(\epsilon)$ is given by
\[
w(\epsilon) = \epsilon_1 w_1 + \cdots + \epsilon_t w_t,
\]
where $\epsilon_1, \dots , \epsilon_n \in \{ 1,-1\}$.

\end{definition}

\begin{remark}
When $p=2$ or $3$, Definition \ref{def:wtsys} is not entirely satisfactory and we have to use a modified set of axioms. See Section \ref{sec:smallp} for the modified definitions.
\end{remark}

Define an equivalence relation $\sim$ on pairs $(a,b) \in \mathbb{Z}_p^*$ by declaring that $(a,b) \sim (b,a) \sim (-a,-b) \sim (-b,-a)$. It is straightforward to check that the values of $a^2+b^2$ and $ab$ completely determine $(a,b)$ up to equivalence. If $(a,b,c,\delta,w)$ is a weight system, then replacing any pair $(a_i,b_i)$ by an equivalent pair still gives a weight system.

We have proven that for a prime $p \ge 7$, a $\mathbb{Z}_p$-action on a definite $4$-manifold $X$ with $H_1(X ; \mathbb{Z}) = 0$ gives rise to a weight system. In this case, $\delta_j = [\Sigma_j]^2$ and the weight vectors $w_a$ have the form $w_a = \widehat{e}_a/x$, where $\widehat{e}_1, \dots , \widehat{e}_t$ are lifts of $e_1, \dots e_t$, an orthonormal basis for the $T^t$ summand of $H^2(X ; \mathbb{Z}) \cong T^t \oplus R^r$. We write the weight vectors as
\[
w_a = ((w_a)_1, \dots , (w_a)_n ; (v_a)_1 + (u_a)_1 \tau , \dots (v_a)_m + (u_a)_m \tau ).
\]
Then $\{ (w_a)_i , (v_a)_j \}$ are the weights of the equivariant line bundle $E_a$ corresponding to $\widehat{e}_a$ and the $(u_a)_j$ are given by $(u_a)_j = \langle e_a , [\Sigma_j] \rangle$.

\begin{lemma}
Let $(a,b,c,\delta,w_a)$ be a weight system of type $(n,m)$. For any $a$, replacing $w_a$ by $w_a + 1$ gives another weight system.
\end{lemma}
\begin{proof}
This is straightforward to check.
\end{proof}

We will sometimes refer to the replacement of a weight $w_a$ by $w_a + t_a \cdot 1$, where $t_a \in \mathbb{Z}_p$ as {\em weight shifting}.

We introduce some notation for weight systems that will be useful when considering the connected sum operation described below. Given a weight system $(a,b,c,\delta,w_a)$ of type $(n,m)$, we let $I = \{1, \dots , n\}$, $J = \{1, \dots , m\}$, $K = \{ 1 , \dots , t\}$ be the indexing sets for the points, surfaces and weight vectors. We can think of $\{a_i\},\{b_i\}$ as functions $a,b : I \to \mathbb{Z}_p$, $\{c_j\},\{\delta_j\}$ can be viewed as functions $c, \delta : J \to \mathbb{Z}_p$ and the weights vectors $w_a$ can be viewed as a pair of functions $(w|_I , w|_J)$, where $w|_I : I \to \mathbb{Z}_p$ and $w|_J : J \to \mathbb{Z}_p[\tau]/(\tau^2)$.

Let $(a,b,c,\delta,w_a)$ be a weight system of type $(n,m)$. Fix an index $i \in I$. Then by shifting each weight vector $w_a$ by a multiple of $1$, we can assume that $(w_a)_i = 0$ for each $a \in K$. In this case we think of $i$ as a ``basepoint" and we have normalised the weight vectors to have weight zero over the basepoint. Similarly, if $j \in J$ then by shifting each weight vector $w_a$, we can assume that $(w_a)_j \in \mathbb{Z}_p[\tau]/(\tau^2)$ is a multiple of $\tau$.

Next we consider a connected sum operation. There are two cases corresponding to connected sum at an isolated fixed point and connected sum at a non-isolated fixed point. 

Consider the isolated point case. We will call this a connected sum of type 1. Let $(a_i(1),b_i(1),c_i(1),\delta_i(1),w_a(1)), (a_i(2),b_i(2),c_i(2),\delta_i(2),w_a(2))$ be weight systems of types $(n_1,m_1)$ and $(n_2,m_2)$. Suppose there are indices $i_1 \in \{1, \dots , n_1\}$, $i_2 \in \{1, \dots , n_2\}$ such that:
\[
a_{i_1}(1)b_{i_1}(1) = -a_{i_2}(2)b_{i_2}(2), \quad a_{i_1}^2(1) + b_{i_1}^2(1) = a_{i_2}(2)^2 + b_{i_2}^2(2).
\]
Equivalently, $e(1)_{i_1} = -e(1)_{i_2}$ and $\rho(1)_{i_1} = \rho(2)_{i_2}$. We assume that $(a_i(1)$, $b_i(1)$, $c_i(1)$, $\delta_i(1)$,$w_a(1))$ is based at $i_1$ and $(a_i(2),b_i(2),c_i(2),\delta_i(2),w_a(2))$ is based at $i_2$. Set $I_1 = \{1, \dots , n_1\}$, $I_2 = \{1, \dots , n_2\}$, $J_1 = \{1, \dots , m_1\}$, $J_2 = \{1 , \dots , m_2\}$, $K_1 = \{1, \dots , t_1\}$, $K_2 = \{1 , \dots , t_2\}$ ($t_i = n_i + 2m_i - 2$ for $i=1,2$). Define new index sets:
\begin{align*}
I &= ((I_1 \setminus \{i_1\} ) \times \{1\}) \cup ((I_2 \setminus \{i_2\} ) \times \{2\}), \\
J &= (J_1 \times \{1\}) \cup (J_2 \times \{2\}), \\
K &= (K_1 \times \{1\} ) \cup (K_2 \times \{2\}).
\end{align*}

Then $|I| = n$, $|J| = m$, $|K| = t$, where $n = n_1 + n_2 - 2$, $m = m_1 + m_2$ and $t = t_1 + t_2$ where $t = n+2m-2$. We will construct a weight system of type $(n,m)$. The vectors $a(1), a(2)$ can be combined to form a vector $a = (a_i)_{i \in I}$. Similarly we obtain $b,c,\delta$. The $w_a$ are defined as follows. If $a = (a_1 , 1)$ for some $a_1 \in K_1$, then we take
\begin{align*}
(w_a)_{(i,1)} &= (w_{a_1})_{i} \text{ for } i \in I_1 \setminus \{i_1\}, \\
(w_a)_{(i,2)} &= 0 \text{ for } i \in I_2 \setminus \{i_2\}, \\
(w_a)_{(j,1)} &= (w_{a_1})_{j} \text{ for } j \in J_1, \\
(w_a)_{(j,2)} &= 0 \text{ for } j \in J_2.
\end{align*}

Similarly, if $a = (a_2 , 2)$ for some $a_2 \in K_2$, then we take
\begin{align*}
(w_a)_{(i,1)} &= 0 \text{ for } i \in I_1 \setminus \{i_1\}, \\
(w_a)_{(i,2)} &= (w_{a_2})_{i} \text{ for } i \in I_2 \setminus \{i_2\}, \\
(w_a)_{(j,1)} &= 0 \text{ for } j \in J_1, \\
(w_a)_{(j,2)} &= (w_{a_2})_{j} \text{ for } j \in J_2.
\end{align*}

Then it is straightforward to check that this defines a weight system (which need not be based).

Now we consider connected sum for a non-isolated point. So we have to select surfaces on each side on which to perform the connected sum. That is, we suppose there are indices $j_1 \in J_1$, $j_2 \in J_2$ such that
\[
c_{j_1}(1) = c_{j_2}(2).
\]
We assume $(a_i(1),b_i(1),c_i(1),\delta_i(1),w_a(1))$ is based at $j_1$ and $(a_i(2),b_i(2),c_i(2),\delta_i(2),w_a(2))$ is based at $i_2$. Set
\begin{align*}
I &= (I_1 \times \{1\} ) \cup (I_2 \times \{2\} ), \\
J &= (J_1 \times \{1\} ) \cup ((J_2 \setminus \{j_2\} ) \times \{2\}), \\
K & = (K_1 \times \{1\}) \cup (K_2 \times \{2\}).
\end{align*}

The vectors $a,b,c$ are obtained from $a(1),a(2),b(1),b(2),c(1),c(2)$ in the obvious way. For $\delta$, we set $\delta_{(j,1)} = \delta_j(1)$ if $j \in J_1 \setminus \{j_1\}$, $\delta_{(j,2)} = \delta_j(2)$ if $j \in J_2$ and $\delta_{(j_1,1)} = \delta_{j_1}(1) + \delta_{j_2}(2)$. If $a = (a_1 , 1)$ for some $a_1 \in K_1$, then we set
\begin{align*}
(w_a)_{(i,1)} &= (w_{a_1})_{i} \text{ for } i \in I_1, \\
(w_a)_{(i,2)} &= 0 \text{ for } i \in I_2, \\
(w_a)_{(j,1)} &= (w_{a_1})_{j} \text{ for } j \in J_1, \\
(w_a)_{(j,2)} &= 0 \text{ for } j \in J_2 \setminus \{j_2\}.
\end{align*}

If $a = (a_2 , 2)$ for some $a_2 \in K_2$, then we set
\begin{align*}
(w_a)_{(i,1)} &= 0 \text{ for } i \in I_1, \\
(w_a)_{(i,2)} & = (w_{a_2})_i \text{ for } i \in I_2, \\
(w_a)_{(j,1)} &= 0 \text{ for } j \in J_1 \setminus \{j_1\}, \\
(w_a)_{(j_1,1)} &= (w_{a_2})_{j_2}, \\
(w_a)_{(j,2)} &= (w_{a_2})_j \text{ for } j \in J_2 \setminus \{j_2\}.
\end{align*}

Once again, it is straightforward to verify that this gives a weight system.

\begin{example}\label{ex:cp21}
Let $a,b \in \mathbb{Z}_p$ with $a,b \neq 0$ and $a \neq b$. Then we have a weight system of type $(3,0)$ given by
\begin{align*}
(a_1 , b_1 ) &= (a,b), \\
(a_2 , b_2) &= (b-a,-a), \\
(a_3,b_3) &= (a-b,-b), \\
w = (w_0,w_1,w_2) &= ( 0 , a , b ).
\end{align*}

This is the weight system associated to the linear $\mathbb{Z}_p$-action on $\mathbb{CP}^2$ given by $[x,y,z] \mapsto [x , \omega^a y , \omega^b z]$.
\end{example}

\begin{example}\label{ex:cp22}
Let $c \in \mathbb{Z}_p$ with $c \neq 0$. Then we have a weight system of type $(1,1)$ given by
\begin{align*}
(a_1,b_1) &= (c,c), \\
c_1 &= c, \\
\delta_1 &= 1, \\
w &= ( -c ; \tau ).
\end{align*}

This is the weight system associated to the linear $\mathbb{Z}_p$-action on $\mathbb{CP}^2$ given by $[x,y,z] \mapsto [x , \omega^c y , \omega^c z]$.

\end{example}

We will prove that every weight system is given by a connected sum of weights systems of the types given by Examples \ref{ex:cp21} and \ref{ex:cp22}. We will prove this by induction on $t = n+2m-2$. First consider the base case $t=1$.

\begin{proposition}\label{prop:t1}
Any weight system of type $(n,m)$ with $n+2m = 3$ is the weight system of a linear $\mathbb{Z}_p$-action on $\mathbb{CP}^2$.
\end{proposition}
\begin{proof}
There are two cases: $(n,m) = (3,0)$ and $(n,m) = (1,1)$.

{\bf Case $(n,m) = (3,0)$.} The data is $(a_0,b_0), (a_1,b_1), (a_2,b_2), w = (w_0, w_1, w_2)$. We shift so that $w_0 = 0$. Then $\rho = ( p_0 , p_1 , p_2)$ where $\rho_i = a_i^2 + b_i^2$, $e = (e_0 , e_1 , e_2)$ where $e_i = a_i b_i$. For convenience, we set $\alpha_i = e_i^{-1}$. If $f = (f_0 , f_1 , f_2)$, then 
\[
\langle f \rangle = \alpha_0 f_0 + \alpha_1 f_1 + \alpha_2 f_2.
\]
From $(W2)$ we get
\[
\alpha_0 + \alpha_1 + \alpha_2 = 0.
\]
From $(W3)$ we get
\[
\alpha_1 w_1 + \alpha_2 w_2 = 0.
\]
From $(W4)$ we get
\[
\alpha_1 w_1^2 + \alpha_2 w_2^2 = 1.
\]

These equations can be written as
\[
\left[ \begin{matrix} 1 & 1 & 1 \\ 0 & w_1 & w_2 \\ 0 & w_1^2 & w_2^2 \end{matrix} \right] \left[ \begin{matrix} \alpha_0 \\ \alpha_1 \\ \alpha_2 \end{matrix} \right] = \left[ \begin{matrix} 0 \\ 0 \\ 1 \end{matrix} \right].
\]

The determinant of the $3 \times 3$ matrix is $w_1 w_2 (w_1 - w_2)$. We claim this is non-zero. First, if $w_1 = 0$, then from $\alpha_1 w_1 + \alpha_2 w_2 = 0$, we get $\alpha_2 w_2 = 0$. But $\alpha_2 \neq 0$, so $w_2 = 0$. Similarly, if $w_2 = 0$ then $w_1 = 0$. But $w_1 = w_2 = 0$ is impossible, since $\alpha_1 w_1^2 + \alpha_2 w_2^2 = 1$. So $w_1,w_2 \neq 0$. If $w_1 = w_2$, then we would get $w_1(\alpha_1 + \alpha_2) = 0$, so $\alpha_1 + \alpha_2 = 0$ and $1 = w_1^2( \alpha_1 + \alpha_2 ) = 0$, a contradiction. So the above linear system has a unique solution which is easily shown to be
\[
(\alpha_0 , \alpha_1 , \alpha_2) = \left( \frac{1}{w_1 w_2} , \frac{1}{w_1(w_1-w_2)} , \frac{1}{w_2(w_2-w_1)} \right).
\]

From $(W5)$, $(W6)$, we get
\[
\langle \rho \rangle = 3, \quad \langle pw \rangle = \langle w^3 \rangle.
\]
These two equations determine $\rho$ up to the orthogonal complement of $\{1,w\}$, which is spanned by $1$. So we can write $\rho = \rho' + t \cdot 1$ for some $t \in \mathbb{Z}_p$, where $\rho' = (0,\rho'_1 , \rho'_2)$ and $\rho'$ satisfies $\langle \rho' \rangle = 3$, $\langle \rho' w \rangle = \langle w^3 \rangle$. A little algebra shows that
\[
\rho'_1 = w_1^2 - 2w_1 w_2, \quad \rho'_2 = w_2^2 - 2w_1 w_2.
\]

Then we use $(W7)$ to solve for $t$. We have
\[
15 \langle w^4 \rangle - 30 \langle w^2( \rho' + t ) \rangle + 7\langle ({\rho'}^2 + 2t\rho' + t^2) \rangle - 4\langle e^2 \rangle = 0.
\]
One finds $\langle w^4 \rangle = w_1^2 + w_1 w_2 + w_2^2$, $\langle w^2 \rho' \rangle = w_1^2 - w_1 w_2 + w_2^2$, $\langle {\rho'}^2 \rangle = w_1^2 + w_2^2 - 7w_1w_2$ and $\langle e^2 \rangle = w_1^2 + w_2^2 - w_1w_2$. Substituting, we find $12t = 12(w_1^2 + w_2^2)$. So if $p \neq 2,3$, we get $t =w_1^2 + w_2^2$ and hence
\begin{align*}
\rho &= ( w_1^2 + w_2^2 , w_1^2 + (w_1-w_2)^2 , w_2^2 + (w_1 - w_2)^2 ), \\
e &= (w_1 w_2 , w_1(w_1-w_2) , w_2(w_2-w_1) ).
\end{align*}
In fact the above formula for $\rho$ is valid even when $p=2,3$ because when $p=2$ we have $\rho = (0,0,0)$ and when $p = 3$, we have $\rho = (-1,-1,-1)$.

The vectors $\rho$ and $e$ completely determine the pairs $(a_0,b_0), (a_1,b_1), (a_2,b_2)$ up to equivalence. Up to equivalence, we find
\[
(a_0,b_0) = (w_1,w_2), \quad (a_1,b_1) = (-w_1 , w_2-w_1), \quad (a_2,b_2) = (-w_2 , w_1 - w_2).
\]
These are the weights of the linear action on $\mathbb{CP}^2$ given by $[x,y,z] \mapsto [x , \omega^{w_1} y , \omega^{w_2}z]$.

{\bf Case $(n,m) = (1,1)$.} The data is $(a,b), (c,\delta), w = (w_0 ; w_1 + w_2 \tau)$. We shift so that $w_1 = 0$. If $f \in H_{1,1}$ is written as $f = (f_0 ; f_1 + f_2 \tau )$, then 
\[
\langle f \rangle = \alpha_0 f_0 + \alpha_1 f_1 + \alpha_2 f_2,
\]
where
\[
(\alpha_0 , \alpha_1 , \alpha_2 ) = \left( \frac{1}{ab} , -\frac{\delta}{c^2} , \frac{1}{c} \right).
\] 

We have
\[
\rho = ( \rho_0 , c^2 + 2c \delta \tau )
\]
where we set $\rho_0 = a^2 + b^2$, and
\[
e = ( ab ; 2c \tau ).
\]
From $(W2)$, we get $\alpha_0 + \alpha_1 = 0$, so $\alpha_1 = -\alpha_0$.

From $(W3)$, we get
\[
\alpha_0 w_0 + \alpha_2 w_2 = 0
\]
and from $(W4)$, we deduce that $w_0 \neq 0$ and
\[
\alpha_0 = \frac{1}{w_0^2}.
\]
Hence $\alpha_1 = -1/w_0^2$ and $\alpha_2 = -1/(w_0w_2)$. So we have
\[
(\alpha_0 , \alpha_1 , \alpha_2) = \left( \frac{1}{w_0^2} , -\frac{1}{w_0^2} , -\frac{1}{w_0w_2} \right).
\]
From $(W5)$, $(W6)$, we get $\langle \rho \rangle = 3$, $\langle \rho w \rangle = \langle w^3 \rangle = w_0$. These two equations determine $\rho$ up to addition of a multiple of $1$, so if we write $\rho = \rho' + t$ where $\rho' = (\rho'_0 ;  \rho'_2 \tau )$ then $\rho'$ is uniquely determined. We find
\[
\rho'_0 = w_0^2, \quad \rho'_2 = -2w_0w_2.
\]

Now we use $(W7)$ to solve for $t$. We have
\[
15 \langle w^4 \rangle - 30 \langle w^2( \rho' + t ) \rangle + 7\langle ({\rho'}^2 + 2t\rho + t^2) \rangle - 4\langle e^2 \rangle = 0.
\]
We have $\langle w^4 \rangle = w_0^2$, $\langle w^2 \rho' \rangle = w_0^2$, $\langle {\rho'}^2 \rangle = w_0^2$, $\langle e^2 \rangle = w_0^2$. This gives $12t = 12w_0^2$. If $p \neq 2,3$ we get $t = w_0^2$ and thus
\[
\rho = (2w_0^2 ;  w_0^2 - 2w_0 w_2 \tau).
\]
In fact, the above formula for $\rho$ is easily seen to hold even when $p=2,3$. When $p=2$, we have $\rho = (0 ; 1 ) = (2w_0^2 , w_0^2 - 2w_0 w_2 \tau)$. When $p=3$, we have $\rho = ( -1 ; 1 + 2c \delta \tau)$. But $-\delta/c^2 = \alpha_1 = -1/w_0^2$, so $\delta = 1$ (as $c^2 = w_0^2 = 1$ for $\mathbb{Z}_3$). Also $-1/(w_0w_2) = \alpha_2 = 1/c$ gives $c = -w_0 w_2$, hence $\rho = (-1 ; 1+2c\delta \tau) = ( 2w_0^2 ; 1 - 2w_0w_2 \tau)$.

From $ab = \alpha_0^{-1} = w_0^2$, $a^2 + b^2 = 2w_0^2$, we deduce that $(a,b) = (w_0,w_0)$ up to equivalence. 

From $c^2 + 2c\delta \tau = w_0^2 - 2w_0 w_2 \tau$ we deduce that $c^2 = w_0^2$ and $c\delta = -w_0 w_2$. From $\alpha_0 = -\alpha_1$, we deduce that $(1/w_0^2) = \delta/c^2 = \delta/w_0^2$. Hence $\delta = 1$ and $c = -w_0 w_2$. Since $c^2 = w_0^2$ this means $w_2 = \pm 1$. We have shown that
\[
(a,b,c,\delta) = (w_0 , w_0 , -w_2 w_0 , 1 )
\]
where $w_2 = \pm 1$. These are the weights of the linear action on $\mathbb{CP}^2$ given by $[x,y,z] \mapsto [x, \omega^{w_0}y , \omega^{w_0}z]$. The choice of sign $w_2 = 1$ or $w_2 = -1$ correspond to the two possible choices of orientation for the surface component of the fixed point set.

\end{proof}

\section{Decomposition of weight systems}\label{sec:decomp}

Let $(a,b,c,\delta,w)$ be a weight system of type $(n,m)$ with point index set $I$, surface index set $J$ and weight index set $K$. We assume $|K| > 1$ and write $K = \{0,1,\dots , t\}$ so that $|K| = t+1$ with $t > 0$. With this convention we have $n+2m = t-1$.

Let $e_0 \in H_{n,m}$ be given by $e_0 = (1,0 , \dots , 0 ; 0 , \dots ,0 )$ or $e_0 = (0, \dots , 0 ; \tau , 0 , \dots , 0)$. Then $e_0^2 = e_0$ in the first case and $e_0^2 = 0$ in the second case. We have $\langle e_0 \rangle = \alpha_0$, where $\alpha_0 = 1/(a_1b_1)$ in the first case and $\alpha_0 = 1/c_1$ in the second case. In either case $\alpha_0 \neq 0$. After shifting weights, we can arrange that $w_a e_0 = 0$ for all $a$. This means that the weight system is based at a point in the case $e_0 = (1, 0 , \dots , 0 ; 0 , \dots , 0)$ or based at a surface in the case $e_0 = (0 , \dots , 0 ; \tau , 0 , \dots , 0)$.

The bilinear form $\langle u , v \rangle = \langle uv \rangle$ is easily seen to be non-degenerate. We claim that the classes $1, w_0, \dots , w_t$ are linearly independent. First of all $\langle w_a , w_b \rangle = \delta_{ab}$ shows that $w_0, \dots , w_t$ are independent. If $1 = \lambda_0 w_0 + \cdots + \lambda_t w_t$ for some $\{ \lambda_i\}$, then taking inner product with $w_a$ gives $\lambda_a = 0$ for each $a$. This implies $1 = 0$, a contradiction. Hence there exists $\lambda \in H_{n,m}$ such that $\langle 1 , \lambda \rangle = 1$, $\langle  \lambda , w_a \rangle = 0$ for all $a$. If $p$ is odd, then replacing $\lambda$ by $\lambda + u \cdot 1$ for suitably chosen $u \in \mathbb{Z}_p$, we can also assume that $\langle \lambda , \lambda \rangle = 0$. Then $1, w_0 , \dots , w_t , \lambda$ is a basis for $H_{n,m}$. 

In the case $p=2$, a weight system $W$ must have $m > 0$, by Lemma \ref{lem:surf}. In this case we assume $W$ is based at a surface and we take $\lambda = e_0$. Then it is easily seen that we have $\langle \lambda , 1 \rangle = 1$, $\langle \lambda , \lambda \rangle = 0$, $\langle \lambda , w_a \rangle = 0$ for all $a$.

Since $w_a e_0 = 0$ for all $a$, one sees that the orthogonal complement of $w_0, \dots , w_t$ is spanned by $1$ and $e_0$. Hence $\lambda = u + ve_0$ for some $u,v$. From $\langle \lambda \rangle = 1$, $\langle \lambda , \lambda \rangle = 0$, one solves for $u,v$. We find
\[
\lambda = \frac{1}{\alpha_0}\left( e_0 - \frac{1}{2} \right)
\]
in the case that the weight system is based at a point and
\[
\lambda = \frac{1}{\alpha_0} e_0
\]
in the case that the weight system is based at a surface. Since $e_0^2 = e_0$ in the point case and $e_0^2 = 0$ in the surface case, we find
\[
\lambda^2 = G^2
\]
where we set $G = -1/(2\alpha_0)$ in the point case, or $G = 0$ in the surface case.

\begin{proposition}\label{prop:products}
Let $W = (a,b,c,\delta,w)$ be a based weight system for $p \ge 5$. Then we have:
\begin{align*}
w_a^2 &= \lambda + F_a w_a + \sum_{b | b \neq a} D_{ab} w_b + G, \\
w_a w_b &= D_{ab}w_a + D_{ba}w_b, \\
\lambda w_a &= Gw_a, \\
\lambda^2 &= G,
\end{align*}

where we set
\[
D_{ab} = \langle w_a^2 w_b \rangle, \quad F_a = \langle w_a^3 \rangle.
\]
\end{proposition}
\begin{proof}
We already showed that $\lambda^2 = G$. From $e_0 w_a = 0$ and $\lambda = (1/\alpha_0)e_0 + G$ we deduce that
\[
\lambda w_a = Gw_a.
\]

Expanding $w(\epsilon)$ in $(W6)$ and extracting the $\epsilon_a \epsilon_b \epsilon_c$-term, we get
\[
6\langle w_a w_b w_c \rangle = 0 \text{ for all distinct } a,b,c.
\]
So if $p \ge 5$ then $\langle w_a w_b w_c \rangle = 0$ for distinct $a,b,c$. This means that the expansion of $w_a w_b$ in the basis $1,\lambda, \{ w_c\}_{c \in K}$ contains no $w_c$-component for any $c \neq a,b$. Also $\langle w_a w_b \rangle = 0$ by $(W4)$ and $\langle \lambda w_a w_b \rangle = G\langle w_a w_b \rangle = 0$ since $\lambda w_a = Gw_a$. So there is no $1$ or $\lambda$-component either. Hence we must have
\[
w_a w_b = D_{ab}w_a + D_{ba}w_b
\]
where $D_{ab} = \langle w_a^2 w_b \rangle$.

Lastly, consider the expansion of $w_a^2$ in the basis $1,\lambda, \{w_c\}_{c \in K}$. We have $\langle w_a^2 \rangle = 1$ by $(W4)$, $\langle \lambda w_a^2 \rangle = 0$ because $\lambda w_a = Gw_a$ and $\langle w_a \rangle = 0$ by $(W3)$. Also $\langle w_a^2 w_b \rangle = D_{ab}$ for all $b \neq a$ from the definition of $D_{ab}$. Hence
\[
w_a^2 = \lambda + F_a w_a + \sum_{b | b \neq a} D_{ab} w_b
\]
where $F_a = \langle w_a^3 \rangle$.
\end{proof}

\begin{proposition}\label{prop:dabdba}
Let $(a,b,c,\delta,w)$ be a based weight system where $p \neq 2$. Then $D_{ab}D_{ba} = 0$ for all $a \neq b$. Hence either $D_{ab} = 0$ or $D_{ba} = 0$.
\end{proposition}
\begin{proof}
The $\epsilon_a$-terms in $(W6)$ give:
\begin{align*}
\langle \rho w_a \rangle &= \langle w_a^3 + 3w_a \sum_{b|b \neq a} w_b^2 \rangle \\
&= F_a + 3 \sum_{b | b \neq a} D_{ba}.
\end{align*}

The $\epsilon_a \epsilon_b$-terms in $(W7)$ give:
\begin{align*}
\langle \rho w_a w_b \rangle &= 3 \langle w_a w_b \sum_{c| c\neq a,b} w_c^2 \rangle + \langle w_a^3 w_b + w_b^3 w_a \rangle \\
&= 3\langle (D_{ab}w_a + D_{ba}w_b) \sum_{c | c \neq a,b} w_c^2 \rangle + \langle w_a^2( D_{ab}w_a + D_{ba}w_b) + w_b^2( D_{ab}w_a + D_{ba}w_b) \rangle \\
&= 3\sum_{c| c \neq a,b} (D_{ab} D_{ca} + D_{ba} D_{cb} ) + D_{ab}F_a + D_{ba}F_b + 2D_{ab}D_{ba}.
\end{align*}

On the other hand, we have
\begin{align*}
\langle \rho w_a w_b \rangle &= \langle \rho(D_{ab}w_a + D_{ba}w_b) \rangle \\
&= D_{ab}( F_a + 3D_{ba} + 3 \sum_{c | c \neq a,b} D_{ca} ) + D_{ba}( F_b + 3D_{ab} + 3 \sum_{c | c \neq ab} D_{cb} ) \\
&= D_{ab}F_a + D_{ba}F_b + 6 D_{ab}D_{ba} + 3 \sum_{c | c \neq a,b} ( D_{ab} D_{ca} + D_{ba} D_{cb} ).
\end{align*}

Equating the two expressions gives
\[
4 D_{ab} D_{ba} = 0.
\]
Hence $D_{ab} D_{ba} =  0$ if $p \neq 2$.

\end{proof}

\begin{lemma}\label{lem:dabc}
For distinct $a,b,c$, we have
\[
D_{ab} D_{ac} = D_{ab}D_{bc} + D_{ac} D_{cb}.
\]
\end{lemma}
\begin{proof}
We expand the product $w_a w_b w_c$ in two ways. First
\begin{align*}
w_a w_b w_c &= w_a( w_b w_c) \\
&= w_a( D_{bc}w_b + D_{cb}w_c ) \\
&= D_{bc}( D_{ab}w_a + D_{ba}w_b) + D_{cb}( D_{ac}w_a + D_{ca}w_c) \\
&= (D_{ab}D_{bc} + D_{ac}D_{cb})w_a + D_{ba}D_{bc}w_b + D_{ca}D_{cb}w_c.
\end{align*}

On the other hand
\begin{align*}
w_a w_b w_c &= (w_a w_b)w_c \\
&= (D_{ab}w_a + D_{ba}w_b)w_c \\
&= D_{ab}(D_{ac}w_a + D_{ca}w_c) + D_{ba}( D_{bc}w_b + D_{cb}w_c) \\
&= D_{ab}D_{ac}w_a + D_{ba}D_{bc}w_b + (D_{ca}D_{ab} + D_{cb}D_{ba})w_c.
\end{align*}

Equating coefficients of $w_a$ gives the result.

\end{proof}

\begin{lemma}\label{lem:cubic}
Every $w_a$ satisfies a cubic equation
\[
w_a^3 - F_a w_a^2 - G_a w_a = 0
\]
where
\[
G_a = 2G + \sum_{b | b \neq a} D_{ab}^2.
\]
\end{lemma}
\begin{proof}
Start with $w_a^2 = \lambda + F_a w_a + \sum_{b | b \neq a} D_{ab} w_b + G$. Multiply both sides by $w_a$ to get
\begin{align*}
w_a^3 &= 2Gw_a + F_a w_a^2 + \sum_{b | b \neq a } D_{ab} w_a w_b \\
&= 2G w_a + F_a w_a^2 + \sum_{b | b \neq a} ( D_{ab}^2 w_a + D_{ab}D_{ba} w_b).
\end{align*}
But since $D_{ab}D_{ba} = 0$ for all $b \neq a$ by Proposition \ref{prop:dabdba}, the result follows.
\end{proof}

Define a partial relation $>$ on the set $K$ by setting $a > b$ if $D_{ab} \neq 0$. Since $D_{ab}D_{ba} = 0$ for all $a \neq b$, we see that $a > b$ and $b > a$ can not both occur. If $a > b$, then $D_{ab} \neq 0$, but $D_{ab}D_{ba} = 0$ and so we must have $D_{ba} = 0$.

\begin{lemma}
If $a>b$ and $b>c$ then $a>c$.
\end{lemma}
\begin{proof}
Suppose $a>b$ and $b>c$ but that $a \ngtr c$. So $D_{ab}, D_{bc} \neq 0$ and $D_{ba} = D_{cb} = D_{ac} = 0$. But from Lemma \ref{lem:dabc}, we have $D_{ab}D_{ac} = D_{ab}D_{bc} + D_{cb} D_{ac}$, which reduces to $0 = D_{ab}D_{bc}$, contradicting $D_{ab}, D_{bc} \neq 0$.
\end{proof}

\begin{lemma}\label{lem:noloop}
For any $r \ge 2$ there does not exist a sequence $a_1, \dots , a_r$ with $a_1 > a_2 > \cdots > a_r > a_1$.
\end{lemma}
\begin{proof}
In the case $r=2$ this would give $a_1 > a_2$ and $a_2 > a_1$, which is impossible. Now let $r > 2$ and assume the result holds for $r-1$. If $a_1 > a_2 > \cdots > a_r > a_1$, then by transitivity, $a_1 > a_2 > a_3$ implies $a_1 > a_3$ and so $a_1 > a_3 > \cdots > a_r > a_1$, which contradicts the inductive hypothesis. So the result follows by induction.
\end{proof}

\begin{lemma}\label{lem:min}
There exists an index $a \in K$ such that $a \ngtr b$ for all $b \neq a$.
\end{lemma}
\begin{proof}
Assume that for every $a \in K$ there exists $b \in K$ with $a > b$. Choose any $a_1 \in K$. Choose $a_2 \in K$ with $a_1 > a_2$. Choose $a_3 \in K$ with $a_2 > a_3$. Continuing in this fashion, we get a sequence $a_1 > a_2 > \cdots > a_r$ where $r$ can be arbitrarily large. But if $r > |K|$ then the sequence $\{ a_1 , a_2 , \dots , a_r \}$ must have an element that appears at least twice. This contradicts Lemma \ref{lem:noloop}.
\end{proof}

By Lemma \ref{lem:min}, there exists an $a \in K$ such that $a \ngtr b$ for all $b \neq a$. After possibly re-indexing, we can assume that $a=0$. Thus $0 \ngtr b$ for all $b \neq 0$. That is, $D_{0b} = 0$ for all $b \neq 0$. Therefore, $w_0 w_b = D_{b0}w_b$ for all $b \neq 0$. Since $D_{0b} = 0$ for all $b \neq 0$, we get
\[
w_0^2 = \lambda + F_0 w_0 + G.
\]
Hence
\begin{align*}
D_{b0}^2 w_b &= w_0^2 w_b \\
&= ( \lambda + F_0 w_0 + G)w_b \\
&= F_0 (w_0 w_b) + 2G w_b \\
&= (F_0 D_{b0} + 2G)w_b.
\end{align*}

Therefore, $D_{b0}$ is a solution of the quadratic equation
\[
x^2 - F_0 x - 2G = 0.
\]
Let $\lambda_1, \lambda_2$ be the two roots of this equation. So for each $b \neq 0$, we have that $D_{b0} = \lambda_1$ or $\lambda_2$ and thus $w_0 w_b = \lambda_1 w_b$ or $w_0 w_b = \lambda_2 w_b$.

\begin{definition}
A block decomposition of $(a,b,c,\delta,w)$ consists of:
\begin{itemize}
\item[(1)]{A decomposition $K = K_1 \cup K_2$ into disjoint non-empty subsets $K_1,K_2$.}
\item[(2)]{For each $a \in K_1$ a constant $t_a \in \mathbb{Z}_p$ such that $w_a w_b = t_a w_b$ for all $b \in K_2$.} 
\end{itemize}

\end{definition}

\begin{lemma}\label{lem:bd1}
Let $(a,b,c,\delta,w)$ be a based weight system. If $|K| > 1$ then a block decomposition exists. 
\end{lemma}
\begin{proof}
Choose $w_0$ as above. Let $\lambda_1, \lambda_2$ be the roots of $x^2 - F_0 x - 2G$. Then for each $a \neq 0$, $w_0 w_a = \lambda_1 w_a$ or $\lambda_2 w_a$. If $\lambda_1 = \lambda_2$, then we can take $K_1 = \{0\}$, $t_0 = \lambda_1$, $K_2 = K \setminus \{0 \}$. If $\lambda_1 \neq \lambda_2$ then we have a decomposition $K = \{0 \} \cup K'_1 \cup K'_2$ where for $i=1,2$, we set $K'_i = \{ a \in K \setminus \{0 \} \; | \; w_0 w_a = \lambda_i w_a \}$. Since $|K| > 1$, we have $|K'_1| > 0$ or $|K'_2| > 0$. Without loss of generality, assume $|K'_2| > 0$. Then we will take $K_1 = \{0\} \cup K'_1$ and $K_2 = K'_2$. Let $a \in K'_1$, $b \in K_2$. Then $w_0 w_a w_b = (w_0 w_a)w_b = \lambda_1 w_a w_b$. Also $w_0 w_a w_b = w_a( w_0w_b) = \lambda_2 w_a w_b$, so $(\lambda_1 - \lambda_2)w_a w_b = 0$. Hence $w_a w_b = 0$, since $\lambda_1 \neq \lambda_2$. Finally, if $b \in K_2$, then $w_0 w_b = \lambda_2 w_b$. So we have a block decomposition with $t_0 = \lambda_2$, $t_a = 0$ for $a \in K'_1$.
\end{proof}

\begin{lemma}\label{lem:bd2}
If $K_1,K_2$, $\{ t_a \}_{a \in K_1}$ is a block decomposition of $(a,b,c,\delta,w)$ and $|K_2| > 1$ then there exists a block decomposition $K'_1, K'_2$, $\{ t'_a \}_{a \in K'_1}$ of $(a,b,c,\delta,w)$ with $K'_2$ a proper subset of $K_2$.
\end{lemma}
\begin{proof}
Recall that we defined a relation $>$ on $K$ by $a > b$ if $D_{ab} \neq 0$. Consider the restriction of this relation to $K_2$. Applying the proof of Lemma \ref{lem:min} to $(K_2 , > )$, we find an element $c \in K_2$ such that $c \ngtr b$ for all $b \in K_2$, $b \neq c$. That is, $D_{cb} = 0$ for all $b \in K_2 \setminus \{c\}$. Then
\[
w_c^2 = \lambda + F_c w_c + \sum_{a \in K_1} D_{ca}w_a + G.
\]
If $b \in K_2 \setminus \{c\}$, then $D_{cb} = 0$, so $w_c w_b = D_{bc}w_b$. Then
\begin{align*}
D_{bc}^2 w_b &= w_c^2 w_b \\
&= (\lambda + F_c w_c + \sum_{a \in K_1} D_{ca}w_a + G)w_b \\
&= 2Gw_b + F_c D_{bc}w_b + \sum_{a \in K_1} D_{ca} t_a w_b.
\end{align*}
Hence $D_{bc}$ is a root of
\[
x^2 -F_c x + (\sum_{a \in K_1} D_{ca}t_a + 2G).
\]
Thus, letting $\lambda_1,\lambda_2$ denote the roots of this quadratic, we have that $w_c w_b = \lambda_1 w_b$ or $\lambda_2 w_b$ for every $b \in K_2 \setminus \{c\}$.

If $\lambda_1 = \lambda_2$, we will take $K'_1 = K_1 \cup \{c\}$, $K_2' = K_2 \setminus \{c\}$. This is a block decomposition with $t'_a = t_a$ for $a \in K_1$ and $t'_c = \lambda_1$.

If $\lambda_2 \neq \lambda_2$, let $J_i = \{ b \in K_2 \setminus \{c\} \; | \; w_c w_b = \lambda_i w_b\}$ for $i=1,2$. Then $K_2 = \{c\} \cup J_1 \cup J_2$. Since $|K_2| > 1$ we can assume without loss of generality that $|J_2| > 0$. We will set $K'_1 = K_1 \cup \{c\} \cup J_1$ and $K'_2 = J_2$. If $b_1 \in J_1, b_2 \in J_2$, then $w_c w_{b_1} w_{b_2} = \lambda_1 w_{b_1}w_{b_2} = \lambda_2 w_{b_1} w_{b_2}$, so $w_{b_1}w_{b_2} = 0$. So we have a block decomposition with $t'_a = t_a$ for $a \in K_1$, $t'_c = \lambda_2$, $t'_b = 0$ for $b \in J_1$.
\end{proof}

\begin{proposition}\label{prop:w0}
Let $(a,b,c,\delta,w)$ be a weight system. Then after shifting weights, there exists an index $0 \in K$ such that $w_0 w_a = 0$ for all $a \neq 0$.
\end{proposition}
\begin{proof}
If $|K|=1$ there is nothing to prove, so assume $|K| > 1$. After shifting weights, we can assume the system is based. Then Lemma \ref{lem:bd1} implies a block decomposition $K_1,K_2$, $\{t_a\}_{a \in K_1}$ exists. Repeatedly applying Lemma \ref{lem:bd2}, we can assume the block decomposition has $|K_2| = 1$. So $K_2$ consists of a single element which we denote by $0$. Since this is a block decomposition, we have $w_a w_0 = t_a w_0$ for all $a \neq 0$. That is, $(w_a - t_a)w_0 = 0$. So after replacing each weight $w_a$ by $w_a - t_a$, we have $w_0 w_a = 0$ for all $a \neq 0$.
\end{proof}

Let $(a,b,c,\delta , w)$ be a weight system of type $(n,m)$ with index sets $I,J,K$ and assume $|K| = t+1$ where $t>0$. By Proposition \ref{prop:w0} we may assume, after shifting weights and re-indexing that $K = \{0, 1 , \dots , t\}$ and $w_0w_a = 0$ for all $a \neq 0$. Let $W_0 : H_{n,m} \to H_{n,m}$ be the endomorphism $W_0(x) = w_0 x$. Then $W_0(w_a) = 0$ for all $a \neq 0$. In the case of a based weight system, Lemma \ref{lem:cubic} shows that every $w_a$ satisfies a monic cubic equation. But if $w \in H_{n,m}$ satisfies a monic cubic equation, then so does any shift $w + t \cdot 1$ where $t \in \mathbb{Z}_p$. Applying this to $w_0$, we have
\[
w_0^3 = Aw_0^2 + Bw_0 + C
\]
for some $A,B,C \in \mathbb{Z}_p$. Hence also 
\begin{equation}\label{equ:W0cubic}
W_0^3 = AW_0^2 + BW_0 + C.
\end{equation}
We observe that $1,w_0,w_0^2 , \{ w_a\}_{a\neq 0}$ is a basis for $H_{n,m}$. Writing $H_{n,m} = V_1 \oplus V_2$, where $V_1$ is the span of $1,w_0,w_0^2$ and $V_2$ is the span of $\{ w_a \}_{a \neq 0}$. Then $W_0$ sends $V_i$ to itself for $i=1,2$ and $W_0|_{V_2} = 0$. Equation (\ref{equ:W0cubic}) applied to $w_a$ for any $a \neq 0$ gives $C=0$. The matrix of $W_0|_{V_1}$ with respect to $1,w_0,w_0^2$ is the companion matrix of the cubic $c(x) = x^3 - Ax^2 - Bx$. Every eigenvalue of $W_0$ is a root of $c(x)$.

Recall that $H_{n,m} = \bigoplus_{i \in I} \mathbb{Z}_p \oplus \bigoplus_{j \in J} \mathbb{Z}[\tau]/(\tau^2)$. For $i \in I = \{1,\dots , n\}$, let 
\[
p_i = (0 ,\dots , 0 , 1 , 0 , \dots , 0 ; 0 , \dots , 0),
\]
where the $1$ appears in the $i$-th position. Similarly for $j \in J = \{1, \dots , m\}$, let 
\[
q_j = (0 , \dots , 0 ; 0 , \dots , 0 , 1 , 0 , \dots , 0).
\]
Let $P_i \subseteq H_{n,m}$ denote the subspace spanned by $p_i$ and let $S_j \subseteq H_{n,m}$ denote the subspace spanned by $\{ q_j , \tau q_j \}$. Then $P_i,S_j$ are $W_0$-invariant subspaces of $H_{n,m}$. In particular, $W_0$ must acts on each $P_i$ as multiplication by some eigenvalue of $W_0$, so as multiplication by some root of $c(x)$. In the case of $S_j$, we have that $W_0$ commutes with multiplication by $\tau q_j$ and hence the action of $W_0$ on $S_j$ must be of the form
\[
W_0(q_j) = uq_j + v \tau q_j, \quad W_0( \tau q_j ) = u \tau q_j
\]
for some $u,v \in \mathbb{Z}_p$. In particular, $u$ must be an eigenvalue of $W_0$, hence a root of $c(x)$.

\begin{lemma}\label{lem:roots}
The roots of the cubic $c(x) = x^3 - Ax^2 - Bx$ all lie in $\mathbb{Z}_p$.
\end{lemma}
\begin{proof}
Let $0,u_1,u_2$ be the roots of $x^3 - Ax^2 - Bx$, where $u_1,u_2$ possibly lie in an extension of $\mathbb{Z}_p$. We need to show $u_1,u_2 \in \mathbb{Z}_p$. Since $u_1 + u_2 = A \in \mathbb{Z}_p$ either $u_1$ and $u_2$ both lie in $\mathbb{Z}_p$, or neither do.

Suppose that $u_1$ and $u_2$ do not lie in $\mathbb{Z}_p$. Then the only eigenvalue of $W_0$ that lies in $\mathbb{Z}_p$ is $0$. Hence $W_0|_{P_i} = 0$ for each $i$ and $W_0|_{S_j}$ has the form $W_0(q_j) = v \tau q_j$, $W_0( \tau q_j) = 0$. So $W_0^2 = 0$ and therefore $w_0^2 = W_0^2(1) = 0$. But this would give $\langle w_0^2 \rangle = 0$ which contradicts $(W4)$. So $u_1,u_2$ lie in $\mathbb{Z}_p$.
\end{proof}

Let $0,u_1,u_2$ denote the roots of $c(x)$. By Lemma \ref{lem:roots}, $u_1,u_2 \in \mathbb{Z}_p$. Then for each $i \in I$, $W_0|_{P_i}$ is multiplication by $0,u_1$ or $u_2$. For each $j \in J$, the matrix of $W_0|_{S_j}$ with respect to $q_j , \tau q_j$ has the form
\[
\left[ \begin{matrix} 0 & 0 \\ v & 0 \end{matrix} \right], \quad \left[ \begin{matrix} u_1 & 0 \\ v & u_1 \end{matrix} \right] \text{ or } \left[ \begin{matrix} u_2 & 0 \\ v & u_2 \end{matrix} \right], 
\]
for some $v \in \mathbb{Z}_p$.

\begin{lemma}\label{lem:w0cases}
After re-indexing $I,J$, $w_0$ has one of the following three forms:

{\bf Case 1.}
\[
w_0 = (u_1 , u_2 , 0 , \dots , 0 ; 0 , \dots , 0)
\]
for some $u_1,u_2 \neq 0$. Furthermore $u_1 \neq u_2$.

{\bf Case 2.}
\[
w_0 = (0 , \dots , 0 ; u+v\tau , 0 , \dots , 0)
\]
for some $u,v \neq 0$.

{\bf Case 3.}
\[
w_0 = (u , 0 , \dots , 0 ; v\tau , 0 , \dots , 0)
\]
for some $u,v \neq 0$.

\end{lemma}
\begin{proof}
Recall that $W_0|_{V_2} = 0$. Also, $W_0|_{V_1}$ has characteristic polynomial $c(x)$. Since $0$ is a root of $c(x)$, the kernel of $W_0|_{V_1}$ is at least $1$-dimensional. So $dim( Ker(W_0 ) ) \ge t+1$.

We claim that $c(x)$ has a non-zero root. Assume this is not the case. Then every eigenvalue of $W_0$ is zero. Then $W_0|_{P_i} = 0$ for each $i$ and $(W_0)^2|_{S_j} = 0$ for each $j$. This would imply $W_0^2 = 0$ and hence $w_0^2 = W_0^2(1) = 0$, which contradicts $(W4)$. In particular, $dim( Ker(W_0|_{V_1})) = 1$ or $2$. 

If $dim(Ker(W_0)|_{V_1}) = 2$, then since $W_0$ has a non-zero eigenvalue in $\mathbb{Z}_p$ we must have that $w_0 = (u , 0 , \dots , 0 ; 0 , \dots , 0)$ for some $u \neq 0$. In this case $w_0^2 = uw_0$ and hence $\langle w_0^2 \rangle = u \langle w_0 \rangle = 0$, contradicting $(W4)$. So $dim( Ker(W_0|_{V_2}) ) = 1$ and hence $dim( Ker(W_0) ) = t+1$. It follows that $w_0$ must belong to one of the following four cases (after re-indexing $I,J$):

\begin{itemize}
\item[(1)]{$w_0 = (u_1 , u_2 , 0 , \dots , 0 ; 0 , \dots , 0)$ for some $u_1,u_2 \neq 0$.}
\item[(2)]{$w_0 = (0 , \dots , 0 ; u+v\tau , 0 , \dots , 0)$ for some $u,v$, $u \neq 0$.}
\item[(3)]{$w_0 = (u , 0 , \dots , 0 ; v\tau , 0 , \dots , 0)$ for some $u,v \neq 0$.}
\item[(4)]{$w_0 = (0, \dots , 0 ; v_1 \tau , v_2 \tau , 0 , \dots , 0)$ for some $v_1, v_2 \neq 0$.}
\end{itemize}

In case (4) we have $w_0^2 = 0$ which contradicts $(W4)$, so this case is ruled out. In case (1) we must have $u_1 \neq u_2$ for if $u_1 = u_2$, then $w_0^2 = u_1 w_0$ and so $\langle w_0^2 \rangle = u_1 \langle w_0 \rangle = 0$, which contradicts $(W4)$. Similarly in case (2) we must have $v \neq 0$ otherwise $w_0^2 = u w_0$.

\end{proof}

\begin{theorem}\label{thm:wtdec}
Let $W = (a,b,c,\delta,w)$ be a weight system over $\mathbb{Z}_p$, $p \ge 7$ of type $(n,m)$ with $n+2m = t-1$ and $t > 0$. Then after shifting weights, $W$ is a connected sum of weight systems $W_1, W_2$ of types $(3,0), (n-1,m)$, or $(1,1), (n+1,m-1)$, or $(1,1), (n-1,m)$.
\end{theorem}
\begin{proof}
After weight shifting one of the three cases in Lemma \ref{lem:w0cases} must hold. We consider each case separately.

{\bf Case 1.} We have $w_0 = (u_1 , u_2 , 0 , \dots , 0 ; 0 , \dots , 0 )$ for some $u_1,u_2 \neq 0$, $u_1 \neq u_2$. Of course this implies $n \ge 2$. We will decompose $(a,b,c,\delta,w)$ into a connected sum of weight systems $(a(1),b(1),c(1),\delta(1),w(1))$, $(a(2),b(2),c(2),\delta(2),w(2))$ of type $(3,0)$ and $(n-1,m)$ respectively.

For the weight system of type $(3,0)$ we will take $I(1) = \{ 0,1,2 \}$, $K(1) = \{0\}$, $(a(1)_0,b(1)_0) = (u_1,u_2)$, $(a(1)_1,b(1)_1) = (a_1,b_1)$, $(a(1)_2 , b(1)_2) = (a_2 , b_2)$, $w_0(1) = (0,u_1,u_2)$.

For the weight system of type $(n-1,m)$ we will take $I(2) = \{ 0 , 3, 4 , \dots , n\}$, $J(2) = J$, $K(2) = \{ 1 , 2 , \dots , m\}$, $(a(2)_0,b(2)_0) = (u_1, -u_2)$, $( a(2)_i , b(2)_i ) = (a_i,b_i)$ for $3 \le i \le n$, $c(2)_j = c_j$, $\delta(2)_j = \delta_j$ for all $j \in J$. To define $w(2)_a$ for $a \in K(2)$, first note that $w_0 w_a = 0$, which implies $w_a$ has the form
\[
w_a = ( 0 , 0 , (w_a)_3 , (w_a)_4 , \dots , (w_a)_n ; (w'_a)_1 , \dots , (w'_a)_m )
\]
for some $(w_a)_i \in \mathbb{Z}_p$, $(w'_a)_j \in \mathbb{Z}_p[\tau]/(\tau^2)$. We then set
\[
w(2)_a = ( 0 , (w_a)_3 , (w_a)_4 , \dots , (w_a)_n ; (w'_a)_1 , \dots , (w'_a)_m ).
\]

We need to check that $W_1 = \{ (a(1),b(1),c(1),\delta(1),w(1))\}$, $W_2$ = $\{(a(2)$,$b(2)$, $c(2)$,$\delta(2)$,$w(2))\}$ are weight systems. Clearly $(W1)$ holds. Let us denote by $\langle \; \; \rangle_1 : H_{3,0} \to \mathbb{Z}_p$ and $\langle \; \; \rangle_2 : H_{n-1,m} \to \mathbb{Z}_p$ the maps determined by the corresponding weight systems. We also write $e(1),\rho(1) \in H_{3,0}$, $e(2),\rho(2) \in H_{n-1,m}$ for the Euler and Pontryagin classes of the two weight systems. First of all, note that
\[
\langle w_0^k(1) \rangle_1 = \langle w_0^k \rangle \text{ for all } k \ge 1.
\]
In particular, setting $k=1,2$, we see that $W_1$ satisfies $(W3), (W4)$. Similar reasoning shows that $W_2$ satisfies $(W3), (W4)$. Write $\alpha_i = 1/(a_i b_i)$, for $i=0,2,3$. Then for $f = (f_0,f_1,f_2) \in H_{3,0}$, we have
\[
\langle f \rangle_1 = \alpha_0 f_0 + \alpha_1 f_1 + \alpha_2 f_2.
\]
In particular, since $(W3)$, $(W4)$ hold for $W_1$, we get
\begin{align*}
\alpha_1 u_1 + \alpha_2 u_2 &= 0, \\
\alpha_1 u_1^2 + \alpha_2 u_2^2 &= 1.
\end{align*}
Since $u_1 \neq u_2$, this system of equations for $\alpha_1,\alpha_2$ has a unique solution. Together with $\alpha_0 = 1/(u_1 u_2)$, this gives
\[
(\alpha_0 , \alpha_1 , \alpha_2) = \left( \frac{1}{u_1u_2} , \frac{1}{u_1(u_1-u_2)} , \frac{1}{u_2(u_2-u_1)} \right).
\]
Then it follows that $\langle 1 \rangle_1 = \alpha_0 + \alpha_1 + \alpha_2 = 0$, hence $W_1$ satisfies $(W2)$. Noting that $0 = \langle 1 \rangle = \langle 1 \rangle_1 + \langle 1 \rangle_2$, it follows that $W_2$ also satisfies $(W2)$.

Write $w(\epsilon,1) = \epsilon_0 w_0(1)$, $w(\epsilon,2) = \sum_{a \neq 0} \epsilon_a w_a(2)$. Since $w_0 w_a = 0$ for $a \neq 0$, it follows easily that 
\[
\langle w(\epsilon)^k \rangle = \langle w(\epsilon,1)^k \rangle_1 + \langle w(\epsilon,2)^k \rangle_2 \text{ for any } k \ge 1.
\]
Similar reasoning gives that
\[
\langle w(\epsilon)^k \rho \rangle = \langle w(\epsilon,1)^k \rho(1) \rangle_1 + \langle w(\epsilon,2)^k \rho(2) \rangle_2 \text{ for any } k \ge 1.
\]

Hence $(W6)$ for the weight system $W$ gives
\[
\left( \langle w(\epsilon,1)^3 \rangle_1 - \langle w(\epsilon,1) \rho(1) \rangle_1 \right) + \left( \langle w(\epsilon,2)^3 \rangle_2 - \langle w(\epsilon,2) \rho(2) \rangle_2 \right) = 0 \text{ for all } \epsilon.
\]
By considering $\epsilon_0 = 1$ and $\epsilon_0 = -1$, we see that
\[
\langle w(1)^3_0 \rangle_1 - \langle w(1)_0 \rho(1) \rangle_1 = 0
\]
and
\[
\langle w(\epsilon,2)^3 \rangle_2 - \langle w(\epsilon,2) \rho(2) \rangle_2 = 0 \text{ for all } (\epsilon_1 , \dots , \epsilon_t).
\]

So $W_1,W_2$ both satisfy $(W6)$. The same argument does not apply to $(W7), (W8)$. The reason for this is that $(W6)$ only involves odd powers of $w(\epsilon)$ whereas $(W7), (W8)$ involve even powers.

If we write $\rho(1) = ( \rho_0 , \rho_1 , \rho_2)$, then $(W6)$ for $W_1$ gives
\[
\frac{\rho_1}{(u_1-u_2)} + \frac{\rho_2}{(u_2-u_1)} = \frac{u_1^2}{u_1-u_2} + \frac{u_2^2}{u_2-u_1} = u_1 + u_2,
\]
which simplifies to
\begin{equation}\label{equ:p1p2}
\rho_1 - \rho_2 = u_1^2 - u_2^2.
\end{equation}

Consider $(W8)$ for $W$ with $a = 0$. Since $w_0 w_a = 0$ for $a \neq 0$, we have $w_0 w(\epsilon) = w_0^2$. If $w = 2w_0 + w(\epsilon)$, then $w^2 = 8w_0^2 + w(\epsilon)^2$, $w^3 = 26w_0^2 + w(\epsilon)^3$, $w^4 = 80w_0^4 + w(\epsilon)^4$. Substituting into $(W8)$ and using $(W7)$ to simplify, we arrive at
\[
1200 \langle w_0^4 \rangle - 240 \langle w_0^2 \rho \rangle = 5( 13\langle w_0^3 \rangle - \langle w_0 \rho \rangle )^2.
\]
But since $(W6)$ holds for $W_2$, we see that $13 \langle w_0^3 \rangle - \langle w_0 \rho \rangle = 12 \langle w_0^3 \rangle$, so the above simplifies to
\[
1200 \langle w_0^4 \rangle - 240 \langle w_0^2 \rho \rangle = 720 \langle w_0^3 \rangle^2.
\]
Since $p \ge 7$ we can cancel a common factor of $240$ to get
\begin{equation}\label{equ:w04}
5 \langle w_0^4 \rangle - \langle w_0^2 \rho \rangle = 3 \langle w_0^3 \rangle^2.
\end{equation}
We have $\langle w_0^3 \rangle = u_1+u_2$, $\langle w_0^2 \rho \rangle = (u_1\rho_1 - u_2\rho_2)/(u_1-u_2) = \rho_2 + u_1^2 + u_1 u_2$ and $\langle w_0^4 \rangle = u_1^2 + u_1 u_2 + u_2^2$. Substituting into Equation (\ref{equ:w04}) gives $\rho_2 = u_2^2 + (u_1 - u_2)^2$ and hence $\rho_1 = u_1^2 + (u_1 - u_2)^2$. We also have $\rho_0 = a_0^2 + b_0^2 = u_1^2 + u_2^2$, giving
\[
\rho(1) = (\rho_0 , \rho_1 , \rho_2) = \left( u_1^2 + u_2^2 , u_1^2 + (u_1-u_2)^2 , u_2^2 + (u_1 - u_2)^2 \right).
\]
Then a direct calculation shows that $\langle \rho(1) \rangle = 3$, so $W_1$ satisfies $(W5)$. 

Then $(W5)$ applied to $W$ gives
\[
3t + 3 = \langle \rho \rangle = \langle \rho(1) \rangle_1 + \langle \rho(2) \rangle_2 = 3 + \langle \rho(2) \rangle_2,
\]
hence $(W5)$ also holds for $W_2$.

We have shown that:

\begin{align*}
\rho(1) &= ( u_1^2 + u_2^2 , u_1^2 + (u_1 - u_2)^2 , u_2^2 + (u_1 - u_2)^2 ), \\
e(1) &= ( u_1 u_2 , u_1(u_1-u_2) , u_2(u_2-u_1) ), \\
w(1)_0 &= ( 0 , u_1 , u_2 ).
\end{align*}

It follows that (up to equivalence) we have
\begin{align*}
(a_0,b_0) &= (u_1,u_2),\\
(a_1,b_1) &= (-u_1 , u_2-u_1), \\
(a_2,b_2) &= (-u_2 , u_1-u_2).
\end{align*}

This shows that $W_1$ is the weight system of the linear $\mathbb{CP}^2$-action given by $[x,y,z] \mapsto [x, \omega^{u_1}y , \omega^{u_2}z]$. In particular, $W_1$ is a weight system, so $(W7)$ and $(W8)$ hold for $W_1$. Now since $(W7), (W8)$ hold for both $W$ and $W_1$, it follows easily that $(W7), (W8)$ also hold for $W_2$. So $W_2$ is also a weight system and $W$ is easily seen to be the connected sum of $W_1$ and $W_2$.

{\bf Case 2.} We have $w_0 = (0  , \dots , 0 ; u+v\tau , 0 , \dots , 0 )$ for some $u,v \neq 0$. Of course this implies $m > 0$. We will decompose $W$ into a connected sum of weight systems $W_1 = \{ (a(1),b(1),c(1),\delta(1),w(1))\}$, $W_2 = \{(a(2),b(2),c(2),\delta(2),w(2))\}$ of type $(1,1)$ and $(n+1,m-1)$ respectively.

For $W_1$ we will take $I(1) = \{ 0 \}$, $J(1) = \{1\}$, $K(1) = \{0\}$, $(a(1)_0,b(1)_0) = (u,u)$, $c(1)_1 = c_1$, $\delta(1)_1 = \delta_1$, $w(1)_0 = ( 0 ; u + v\tau)$. For $W_2$ we will take $I(2) = \{ 0 , 1 , \dots , n\}$, $J(2) = \{2 , \dots , m\}$ ($J(2)$ is empty if $m=1$), $K(2) = \{1, \dots , t\}$, $(a(2)_0 , b(2)_0) = (u,-u)$, $(a(2)_i , b(2)_i) = (a_i,b_i)$ for $1 \le i \le n$, $c(2)_j = c_j$ and $\delta(2)_j = \delta_j$ for $j \in J(2)$. The weights $w(2)_a$ will be defines as follows. If $a \neq 0$, then $w_0 w_a = 0$, which means $w_a$ has the form
\[
w_a = ( (w_a)_1 , \dots , (w_a)_n ; 0 , (w'_a)_2 , \dots , (w'_a)_m ).
\]
Then we take
\[
w(2)_a = ( 0 , (w_a)_1 , \dots , (w_a)_n ; (w'_a)_2 , \dots , (w'_a)_m ).
\]

We check that $W_1,W_2$ are weight systems. Clearly $(W1)$ holds for $W_1,W_2$. As in case 1 we use notation such as $\langle \; \rangle_i , e(i), \rho(i)$. Arguing as in Case 1, we see that $(W3), (W4), (W6)$ hold for both $W_1,W_2$. Define $\alpha_0 = 1/u^2$, $\alpha_1 = -\delta_1/c_1^2$, $\alpha_2 = 1/c_1$. Then for $f = (f_0 ; f_1 + f_2 \tau) \in H_{1,1}$, we have
\[
\langle f \rangle_1 = \alpha_0 f_0 + \alpha_1 f_1 + \alpha_2 f_2.
\]
Since $(W3)$, $(W4)$ hold for $W_1$, we get
\begin{align*}
u \alpha_1 + v \alpha_2 &= 0, \\
u^2 \alpha_1 + 2uv \alpha_2 &= 1.
\end{align*}

Since $u,v \neq 0$, we can uniquely solve for $\alpha_1, \alpha_2$, giving
\[
(\alpha_0 , \alpha_1 , \alpha_2 ) = \left( \frac{1}{u^2} , -\frac{1}{u^2} , \frac{1}{uv} \right).
\]
Thus, $\langle 1 \rangle_1 = \alpha_0 + \alpha_1 = 0$. So $(W2)$ holds for $W_1$, which implies that it also holds for $W_2$.

Since $\alpha_2 = 1/(uv) = 1/c_1$, this gives $c_1 = uv$. Since $\alpha_1 = -1/u^2 = -\delta_1/c_1^2$, this gives $\delta_1 = c_1^2/u^2 = v^2$.

We have $\rho(1) = ( 2u^2 , c_1^2 + 2c_1 \delta_1 \tau)$. Then from $(W6)$ for $W_1$, we get
\[
uv = c_1\delta_1 = uv \delta_1
\]
hence $\delta_1 = 1$ as $u,v \neq 0$. So then $v^2 = \delta_1 = 1$, so that $v = \pm 1$. We have shown that $W_1$ is of the form
\[
(a(1)_0,b(1)_0) = (u,u),  c(1)_1 = vu, \delta(1)_1 = 1, w(1)_0 = ( 0 , u + v\tau ), \text{ where } v = \pm 1.
\]
This is the weight system corresponding to a linear $\mathbb{CP}^2$-action of the form $[x,y,z] \mapsto [x , \omega^{u}y , \omega^{u}z]$. In particular $W_1$ is a weight system, so it satisfies $(W1)-(W8)$. Then since $W$ and $W_1$ both satisfy $(W1)-(W8)$, it follows that $W_2$ satisfies $(W5),(W7)$ and $(W8)$. We have that $W$ is the connected sum of $W_1$ and $W_2$.

{\bf Case 3.} We have $w_0 = (u , 0 , \dots , 0 ; v\tau , 0 , \dots , 0)$ for some $u,v \neq 0$ (hence $n,m > 0$). We will decompose $W$ into a connected sum of weight systems $W_1 = \{ (a(1),b(1),c(1),\delta(1),w(1))\}$, $W_2 = \{(a(2),b(2),c(2),\delta(2),w(2))\}$ of type $(1,1)$ and $(n-1,m)$ respectively.

For $W_1$ we will take $I(1) = \{ 0 \}$, $J(1) = \{0\}$, $K(1) = \{0\}$, $(a(1)_0,b(1)_0) = (u,u)$, $c(1)_0 = c_1$, $\delta(1)_0 = 1$, $w(1)_0 = ( u ; v \tau )$. For $W_2$ we will take $I(2) = \{2 , \dots , n\}$, $J(2) = \{ 1, \dots , m\}$, $K(2) = \{1, \dots , t\}$, $(a(2)_i , b(2)_i ) = (a_i,b_i)$ for all $i \in I(2)$, $c(2)_j = c_j$ for $j \in J(2)$, $\delta(2)_j = \delta_j$ for $j \in J(2) \setminus \{1\}$ and $\delta(2)_1 = \delta_1 - 1$. The weights $w(2)_a$ will be defined as follows. If $a \neq 0$ then $w_0 w_a = 0$, so $w_a$ has the form
\[
w_a = ( 0 , (w_a)_2 , \dots , (w_a)_n ; t_a \tau , (w'_a)_2 , \dots , (w'_a)_m )
\]
for some $(w_a)_i , t_a \in \mathbb{Z}_p$, $(w'_a)_j \in \mathbb{Z}_p[\tau]/(\tau^2)$. Then we take
\[
w(2)_a = ( (w_a)_2 , \dots , (w_a)_n ; t_a \tau , (w'_a)_2 , \dots , (w'_a)_m ).
\]

We check that $W_1,W_2$ are weight systems. Clearly $(W1)$ holds for $W_1,W_2$. Arguing as in Case 1,2, we see that $(W3), (W4), (W6)$ hold for both $W_1,W_2$. Define $\alpha_0 = 1/u^2$, $\alpha_1 = -1/c_1^2$, $\alpha_2 = 1/c_1$. Then for $f = (f_0 ; f_1 + f_2 \tau) \in H_{1,1}$, we have
\[
\langle f \rangle_1 = \alpha_0 f_0 + \alpha_1 f_1 + \alpha_2 f_2.
\]
From $(W3)$ for $W_1$, we get
\[
c_1 = -uv,
\]
hence
\[
(\alpha_0 , \alpha_1 , \alpha_2 ) = \left( \frac{1}{u^2} , -\frac{1}{u^2 v^2} , -\frac{1}{uv} \right).
\]

From $(W6)$ for $W_1$, we get $u = -c_1v$. Hence $u = -c_1 v = uv^2$ and hence $v^2 = 1$. So $v = \pm 1$. We have shown that $W_1$ is of the form
\[
(a(1)_0,b(1)_0) = (u,u),  c(1)_0 = -vu, \delta(1)_0 = 1, w(1)_0 = ( u ; v\tau ), \text{ where } v = \pm 1.
\]
This is the weight system corresponding to a linear $\mathbb{CP}^2$-action of the form $[x,y,z] \mapsto [x , \omega^{u}y , \omega^{u}z]$. In particular $W_1$ is a weight system, so it satisfies $(W1)-(W8)$. Then since $W$ and $W_1$ are both weight systems it follows that $W_2$ satisfies $(W5),(W7),(W8)$. We have that $W$ is the connected sum of $W_1$ and $W_2$.

\end{proof}

\begin{corollary}
For $p \ge 7$, every weight system is the weight system associated to an equivariant connected sum of linear $\mathbb{Z}_p$-actions on $\mathbb{CP}^2$, where the summands are connected according to a tree.
\end{corollary}
\begin{proof}
Let $W = (a,b,c,\delta,w)$ be a weight system of type $(n,m)$ and set $t = n+2m-2$. We prove the result by induction on $t$. If $t=1$, then the result follows from Proposition \ref{prop:t1}. Now assume $t > 1$. Then according to Theorem \ref{thm:wtdec}, after weight shifting, $W$ is the connected sum of weight systems $W_1,W_2$ of types $(n_1,m_1), (n_2,m_2)$, where $t_1 = n_1 + 2m_1 - 2 = 1$, $t_2 = n_2 + 2m_2 - 2 = t-1$. So $W_1$ is the weight system of a linear action on $\mathbb{CP}^2$ by Proposition \ref{prop:t1} (or alternatively this follow from the proof of Theorem \ref{thm:wtdec}) and by the inductive hypothesis $W_2$ is the weight system associated to an equivariant connected sum of linear $\mathbb{Z}_p$-actions on $\mathbb{CP}^2$, where the summands are connected according to a tree. It follows that the same is true of the weight system $W$.
\end{proof}

Given a $\mathbb{Z}_p$-action on a definite $4$-manifold $X$ with $H_1(X ; \mathbb{Z})$, we obtain a weight system $W = (a,b,c,\delta,w)$. The mod $p$ values of the self-intersections $[\Sigma_j]^2$ are given by the $\delta_j$ and the mod $p$ values of $\langle e_a , [\Sigma]_j \rangle$ are given by $(u_a)_j$, where $w_a = ( \{ (w_a)_i \} ; \{ (v_a)_j + (u_a)_j \tau \} )$. However, with a little more work, we can actually recover the integer values of $[\Sigma_j]^2$ and $\langle e_a , [\Sigma_j] \rangle$.

\begin{proposition}\label{prop:uaj}
Let $W = (a,b,c,\delta,w)$ be a weight system over $\mathbb{Z}_p$ with $p \ge 7$ of type $(n,m)$ with index sets $I,J,K$. Write the weights as $w_a = ( \{ (w_a)_i \} ; \{ (v_a)_j + (u_a)_j \tau \} )$. Then for each $a \in K$, there is at most one $j \in J$ for which $(u_a)_j \neq 0$. Furthermore, in this case $(u_a)_j = \pm 1$.
\end{proposition}
\begin{proof}
First note that the statement of the proposition is invariant under weight shifting. We prove the result by induction on $t = n+2m-2$. When $t=1$ the result is true by Proposition \ref{prop:t1}. Now suppose $t > 1$. Then by Theorem \ref{thm:wtdec}, after weight shifting, $W$ is a connected sum of weight systems $W_1,W_2$. If $W_1,W_2$ have types $(n_1,m_1), (n_2,m_2)$ and we set $t_i = n_i + 2m_i - 2$, then $t = t_1 + t_2$ and hence $t_1,t_2 < t$. If $W_1,W_2$ satisfy the statement of the proposition then it is easily seen that $W$ does as well. So the result follows by induction.
\end{proof}

\begin{theorem}\label{thm:wtdec2}
Let $p \ge 7$. Let $X$ be a closed, smooth, orientable, positive definite $4$-manifold with $H_1(X ; \mathbb{Z}) = 0$. Let $G = \mathbb{Z}_p = \langle g \rangle$ for $p$ a prime act smoothly on $X$. If $p=2$ assume that $H^2(X ; \mathbb{Z})$ has no cyclotomic summands. Then there is a $4$-manifold $X'$ which is an equivariant connected sum of linear $\mathbb{Z}_p$-actions on $\mathbb{CP}^2$ with the following properties:
\begin{itemize}
\item[(1)]{There is an isometry $\phi : H^2(X' ; \mathbb{Z}) \to H^2(X ; \mathbb{Z})$ which respects the $\mathbb{Z}_p$-actions.}
\item[(2)]{The fixed point sets $\{x_1 , \dots , x_n , \Sigma_1 , \dots , \Sigma_m \}$, $\{x'_1 , \dots , x'_n , \Sigma'_1 , \dots , \Sigma'_m\}$ of $X$ and $X'$ have the same number of points and surfaces and the same normal weights: $(a_i , b_i) = (a'_i , b'_i)$ for all $i$, $c_j = c'_j$ for all $j$. Moreover the fixed surfaces represent the same cohomology classes under $\phi$, that is, $[\Sigma_j] = \phi( [\Sigma'_j])$.}
\item[(3)]{The weights of equivariant line bundles on $X$ and $X'$ agree. More precisely, if we set $e'_i = \phi^{-1}(e_i)$ for $i = 1, \dots , t$, then (for suitably chosen lifts $\widehat{e}_a, \widehat{e}'_a$) we have $(w'_a)_i = (w_a)_i$, $(v'_a)_j = (v_a)_j$ for all $a,i,j$.}
\end{itemize}

\end{theorem}
\begin{proof}
Let the action of $\mathbb{Z}_p$ on $H^2(X ; \mathbb{Z})$ consist of $t$ trivial summands and $r$ regular summands. Let $W = (a,b,c,\delta,w)$ be the corresponding weight system. For each $j$, let $K_j = \{ a \in K \; | (u_a)_j \neq 0 \}$. By Proposition \ref{prop:uaj} the sets $K_j$ are disjoint. By the same proposition $(u_a)_j \in \{0,1,-1\}$ for all $a,j$. Let $(\widehat{u}_a)_j \in \mathbb{Z}$ be the unique lift of $(u_a)_j$ to $\mathbb{Z}$ which is valued in $\{0,1,-1\}$. Since $\langle e_a , [\Sigma_j] \rangle = (\widehat{u}_a)_j \; ({\rm mod} \; p)$, we see that
\[
[\Sigma_j] = \sum_{a=1}^t \sigma_{a,j} e_a + \sum_{b = 1}^{r} c_{b,j} f_j
\]
for some $\sigma_{a,j}, c_{b,j} \in \mathbb{Z}$ and moreover $\sigma_{a,j} = (\widehat{u}_a)_j \; ({\rm mod} \; p)$. It follows that
\begin{equation}\label{equ:sigj2}
[\Sigma_j]^2 \ge | K_j| \text{ with equality iff } \sigma_{a,j} = (\widehat{u}_a)_j, c_{b,j} = 0.
\end{equation}

By Theorem \ref{thm:wtdec}, there is a $\mathbb{Z}_p$-action on $t \mathbb{CP}^2$ which is an equivariant connected sum of linear actions on $\mathbb{CP}^2$ and whose weight system is equal to that of $X$. Choose a point $p \in t \mathbb{CP}^2$ and attach a copy of $\mathbb{CP}^2$ at $p, gp, g^2 p , \dots , g^{p-1}p$. Then we can extend the action of $\mathbb{Z}_p$ to $t \mathbb{CP}^2 \# r \mathbb{CP}^2$ such that the $r$ additional copies of $\mathbb{CP}^2$ are cyclically permuted. Repeating this process $r$ times, we get a $\mathbb{Z}_p$-action on $X' = (t+pr)\mathbb{CP}^2$ which has the same weight system as $X$ and such that the action of $\mathbb{Z}_p$ on $H^2(X' ; \mathbb{Z})$ has $t$ trivial summands and $r$ regular summands. Let $\Sigma_1, \dots , \Sigma_m$ be the fixed surfaces in $X$ and let $\Sigma'_1, \dots , \Sigma'_m$ be the fixed surfaces in $X'$. Since both actions have the same weight systems, the collection of sets $\{ K_j \}$ are the same (after possibly re-indexing). It follows that
\[
[\Sigma'_j] = \sum_{a=1}^{t} (\widehat{u}_a)_j e_a
\]
and therefore $[\Sigma'_j]^2 = |K_j|$. Now we apply the $G$-signature theorem to both $X$ and $X'$. Applied to $X$, we obtain
\[
\sigma^{inv}(X) = \frac{\sigma(X)}{p} - 4 \sum_{i=1}^{n} s( u_j , p ) + \frac{p^2-1}{3p} \sum_{j=1}^{m} [\Sigma_j]^2,
\]
where $\sigma^{inv}(X)$ denote the signature of $H^2(X ; \mathbb{Z})^{\mathbb{Z}_p}$, $s(q , p)$ is the Dedekind sum and $u_j = a_j b'_j$, where $b'_j$ is a mutiplicative inverse of $b_j$ in $\mathbb{Z}_p$. Since $X$ is positive definite, we have $\sigma(X) = b_2(X) = t + pr$ and $\sigma^{inv}(X) = t+r$. So the $G$-signature theorem simplifies to
\[
\frac{(p-1)}{p} t = - 4 \sum_{i=1}^{n} s( u_j , p ) + \frac{p^2-1}{3p} \sum_{j=1}^{m} [\Sigma_j]^2.
\]

Similarly, the $G$-signature theorem applied to $X'$ gives
\[
\frac{(p-1)}{p} t = - 4 \sum_{i=1}^{n} s( u_j , p ) + \frac{p^2-1}{3p} \sum_{j=1}^{m} [\Sigma'_j]^2.
\]
It follows that
\[
\sum_{j=1}^m [\Sigma_j]^2 = \sum_{j=1}^m [\Sigma'_j]^2 = \sum_{j=1}^{m} |K_j|.
\]
Combined with (\ref{equ:sigj2}), this can only happen if $[\Sigma_j] = \sum_a (\widehat{u}_a)_j e_a$.

\end{proof}

\section{Small primes}\label{sec:smallp}

In this section we consider the necessary modifications to the proof of Theorems \ref{thm:wtdec} and \ref{thm:wtdec2} that will allow us to extend these results to the cases $p = 2$, $3$ and $5$.

\subsection{Case $p=5$} Consider a $\mathbb{Z}_5$-action on a positive definite $4$-manifold $X$ with $H_1(X ; \mathbb{Z}) = 0$. The results of Section \ref{sec:definite} carry over to this except for Equations (\ref{equ:dirx43}) and (\ref{equ:dirx44}). Thus the fixed point data $(a,b,c,\delta,w)$ satisfies all the conditions of weight system except possibly $(W7)$ and $(W8)$. We will show that $(W7)$ and $(W8)$ still hold for $p=5$.

\begin{lemma}
Let $W = (a,b,c,\delta,w)$ satisfy $(W1)$-$(W6)$ for $p=5$. Then $W$ also satisfies $(W7),(W8)$.
\end{lemma}
\begin{proof}
For $p=5$, $(W7)$ and $(W8)$ both reduce to
\[
2\langle \rho^2 \rangle + \langle e^2 \rangle = 0.
\]

We have $\rho = (\{a_i^2 + b_i^2\} ; \{ c_j^2 + 2c_j \delta_j \tau \} )$. Then since $a_j,b_j,c_j \neq 0$, we have $a_j^4 = b_j^4 = c_j^4 = 1$, so
\[
\rho^2 = ( \{ 2 + 2a_i^2b_i^2 \} ; \{ 1 - c_j^3 \delta_j \tau \} ).
\]
Also $e^2 = ( \{ a_i^2 b_i^2 \} ; \{ 0 \} )$. So
\[
2 \langle \rho^2 \rangle + \langle e^2 \rangle = -\sum_i \frac{1}{a_ib_i} + \sum_j \delta_j c_j^2 = -\langle 1 \rangle = 0.
\]
Hence $(W7)$ and $(W8)$ hold.
\end{proof}

So for $p=5$, we still obtain a weight system. The proofs of Theorem \ref{thm:wtdec}, \ref{thm:wtdec2} carry over to the case $p=5$ except that in Case 1 of the proof of Theorem \ref{thm:wtdec} we needed to cancel a factor of $240$ in order to deduce that $\rho(1)$ was given by
\begin{equation}\label{equ:rho1}
\rho(1) = (\rho_0 , \rho_1 , \rho_2 ) = ( u_1^2 + u_2^2 , u_1^2 + (u_1-u_2)^2 , u_2^2 + (u_1 - u_2)^2 ).
\end{equation}

We will give a proof that Equation (\ref{equ:rho1}) still holds when $p=5$. We already know that $\rho(1) = u_1^2 + u_2^2$ and from $(W6)$ we deduce that $\rho_1 - \rho_2 = u_1^2 - u_2^2$. From the definition of $\rho(1)$, we have
\begin{align*}
a_1^2 + b_1^2 &= \rho_1, \\
a_2^2 + b_2^2 &= \rho_2.
\end{align*}

Since $e(1) = (u_1u_2 , u_1(u_1-u_2) , u_2(u_2-u_1))$, we also have
\begin{align*}
a_1 b_1 &= u_1(u_2-u_2), \\
a_2 b_2 &= u_2(u_1-u_2).
\end{align*}

Let us set $x_1 = a_1/u_1, x_2  = a_2/u_2, y_1 = b_1/u_1, y_2 = b_2/u_2$, $\gamma_1 = \rho_1/u_1$, $\gamma_2 = \rho_2/u_1$ and $\lambda = u_2/u_1$. Then we have
\begin{align*}
x_1^2 + y_1^2 &= \gamma_1, \\
x_2^2 + y_2^2 &= \gamma_2, \\
x_1y_1&= 1-\lambda, \\
x_2y_2 &= \lambda(\lambda-1), \\
\gamma_1 - \gamma_2 &= 1-\lambda^2.
\end{align*}

Since $u_1,u_2 \neq 0$ and $u_1 \neq u_2$, we have $\lambda \neq 0,1$ and so $\lambda \in \{-1,2,-2\}$. We claim that $\gamma_1 = 1 + (1-\lambda)^2$, $\gamma_2 = \lambda^2 + (1-\lambda)^2$ and therefore $\rho_1 = u_1^2 + (u_1 - u_2)^2$, $\rho_2 = u_2^2 + (u_1 - u_2)^2$. To prove the claim we consider separately the three possible values for $\lambda$.

If $\lambda = -1$, then $\gamma_1 = \gamma_2$ and $x_1 y_1 = 2$. So $(x_1,y_1) \in \{ (1,2) , (2,1) , (-1,-2) , (-2,-1) \}$. In all cases we get that $x_1^2 + y_1^2 = 0$ and thus $\gamma_1 = 0 = 1 + (1-\lambda)^2$, $\gamma_2 = 0 = \lambda^2 + (1-\lambda)^2$.

If $\lambda = 2$, then $\gamma_1 - \gamma_2 = 2$ and $x_2 y_2 = 2$. Similar to the previous case, all possible solutions to $x_2 y_2 = 2$ satisfy $x_2^2 + y_2^2 = 0$. Thus $\gamma_2 = 0 = \lambda^2 + (1-\lambda)^2$ and $\gamma_1 = 2 = 1+(1-\lambda)^2$.

If $\lambda = -2$, then $\gamma_1 - \gamma_2 = 2$ and $x_1 y_1 = 3$. So $(x_1,y_1) \in \{ (1,-2) , (-2,1) , (2,-1) , (-1,2) \}$. In all possible cases, we get $x_1^2 + y_1^2 = 0$. So $\gamma_1 = 0 = 1 + (1-\lambda)^2$ and $\gamma_2 = -2 = \lambda^2 + (1-\lambda)^2$.

The rest of the proof of Theorem \ref{thm:wtdec} and also the proof of Theorem \ref{thm:wtdec2} work exactly the same for $p=5$ and so we deduce that Theorem \ref{thm:wtdec2} also holds for $p=5$.

\subsection{Case $p=3$}

Consider a $\mathbb{Z}_3$-action on a positive definite $4$-manifold $X$ with $H_1(X ; \mathbb{Z}) = 0$. The results of Section \ref{sec:definite} carry over to this except for Equations (\ref{equ:dirx33}), (\ref{equ:dirx43}) and (\ref{equ:dirx44}). Thus the fixed point data $(a,b,c,\delta,w)$ satisfies all the conditions of weight system except possibly $(W6), (W7)$ and $(W8)$. We show these continue to hold when $p=3$.

\begin{lemma}
Let $W = (a,b,c,\delta,w)$ satisfy $(W1)$-$(W5)$ for $p=3$. Then $W$ also satisfies $(W6), (W7)$ and $(W8)$.
\end{lemma}
\begin{proof}
For $p=3$, $(W6), (W7), (W8)$ reduce to:
\begin{align*}
\langle w(\epsilon)^3 \rangle - \langle \rho w(\epsilon) \rangle &= 0, \\
\langle \rho^2 \rangle - \langle e^2 \rangle &= 0, \\
\langle \rho^2 \rangle - \langle e^2 \rangle &= -( \langle (2w_a + w(\epsilon))^3 \rangle - \langle \rho (2w_a + w(\epsilon)) \rangle )^2,
\end{align*}

Write $w(\epsilon) = ( \{ w(\epsilon)_a \} ; \{ v(\epsilon)_j + u(\epsilon)_j \tau \})$. Since $c_j \in \mathbb{Z}_3^*$ we have $c_j^2 = 1$ for all $j$. From $(W3)$ we get
\begin{equation}\label{equ:we0}
0 = \langle w(\epsilon) \rangle = \sum_i \frac{w(\epsilon)_i }{a_ib_i} - \sum_j \delta_j v(\epsilon)_j + \sum_j u(\epsilon)_j c_j.
\end{equation}
Since $z^3 = z$ for any $z \in \mathbb{Z}_3$, we see that
\[
w(\epsilon)^3 = ( \{ w(\epsilon)_i \} ; \{ v(\epsilon)_j \})
\]
and thus
\[
\langle w(\epsilon)^3 \rangle = \sum_i \frac{w(\epsilon)_i}{a_ib_i} - \sum_j \delta_j v(\epsilon)_j.
\]
Since $a_i,b_i \neq 0$, we have $a_i^2 = b_i^2 = 1$. So $\rho$ has the form
\[
\rho = ( \{ -1 \} ; \{ 1 - \delta_j c_j \tau \} ).
\]
Hence
\begin{align*}
\langle \rho w(\epsilon) \rangle &= -\sum_i \frac{w(\epsilon)_i}{a_ib_i} - \sum_j \delta_j v(\epsilon)_j + \sum_j ( u(\epsilon)_j - \delta_j c_j v(\epsilon)_j )c_j \\
&= -\sum_i \frac{w(\epsilon)_i}{a_ib_i} + \sum_j \delta_j v(\epsilon)_j + \sum_j u(\epsilon)_j c_j.
\end{align*}

Hence
\[
\langle w(\epsilon)^3 \rangle - \langle \rho w(\epsilon) \rangle = -\sum_i \frac{w(\epsilon)_i}{a_ib_i} + \sum_j \delta_j v(\epsilon)_j - \sum_j u(\epsilon)_j c_j = 0
\]
where the last equality is (\ref{equ:we0}). Hence $(W6)$ holds.

Next, we have
\[
\rho^2 = (\{ 1\} ; \{ 1 + \delta_j c_j \tau \}), \quad e^2 = ( \{1\} ; 0 )
\]
and so
\[
\langle \rho^2 \rangle = \sum_i \frac{1}{a_ib_i} -\sum_j \delta_j + \sum_j \delta_j c_j^2 = \sum_i \frac{1}{a_ib_i} = \langle e^2 \rangle,
\]
which proves $(W7)$.

Having proven $(W7)$, $(W8)$ reduces to:
\[
\langle (2w_a + w(\epsilon)^3 \rangle - \langle \rho ( 2w_a + w(\epsilon) ) \rangle = 0.
\]
But in $\mathbb{Z}_3$ we have $(2w_a + w(\epsilon))^3 = 8w_a^3 + w(\epsilon)^3 = -w_a^3 + w(\epsilon)^3$, so
\begin{align*}
& \langle (2w_a + w(\epsilon)^3 ) \rangle - \langle \rho( 2w_a + w(\epsilon) ) \rangle \\
& \quad = -\langle w_a^3 \rangle + \langle w(\epsilon)^3 \rangle + \langle \rho w_a \rangle - \langle \rho w(\epsilon) \rangle \\
& \quad = -\langle w_a^3 \rangle + \langle \rho w_a \rangle
\end{align*}
where the last equality follows from $(W6)$. Expanding out the $\epsilon_a$-terms in $(W6)$ gives
\[
\langle w_a^3 \rangle + 3 \langle w_a \sum_{b | b \neq a } w_b^2 \rangle - \langle \rho w_a \rangle = 0,
\]
which in $\mathbb{Z}_3$ reduces to
\[
\langle w_a^3 \rangle - \langle \rho w_a \rangle = 0.
\]
Hence $(W8)$ holds.

\end{proof}

So for $p=3$, we still obtain a weight system. The results of Section \ref{sec:decomp} carry over to the case $p=3$ except at two points:
\begin{itemize}
\item[(1)]{We used $(W6)$ to deduce that $6 \langle w_a w_b w_c \rangle = 0$ for distinct $a,b,c$. When $p \neq 2,3$, this implies that $\langle w_a w_b w_c \rangle = 0$. But for $p=3$ we need a separate argument to show that $\langle w_a w_b w_c \rangle = 0$.}
\item[(2)]{In the proof of Theorem \ref{thm:wtdec} Case 1, we cancelled a factor of $240$ in order to deduce that $\rho(1) = (u_1^2 + u_2^2 , u_1^2 + (u_1- u_2)^2 , u_2^2 + (u_1 - u_2)^2)$. But since $3$ divides $240$, we need a separate argument for this step.}
\end{itemize}

Point (2) is by far the easiest to deal with. We have that $\rho(1) = ( u_1^2 + u_2^2 , a_1^2 + b_1^2 , a_2^2 + b_2^2 )$. But since $u_1,u_2,a_1,a_2,b_1,b_2 \in \mathbb{Z}_3$ are all non-zero we have that $u_1^2 = u_2^2 = a_1^2 = a_2^2 = b_1^2 = b_2^2 = 1$ and thus $\rho(1) = (2,2,2)$. On the other hand, since $u_1,u_2 \neq 0$ and $u_1 \neq u_2$ we also have that $(u_1^2 + u_2^2 , u_1^2 + (u_1-u_2)^2 , u_2^2 + (u_1 - u_2)^2 ) = (2,2,2)$. So $\rho(1) = (u_1^2 + u_2^2 , u_1^2 + (u_1- u_2)^2 , u_2^2 + (u_1 - u_2)^2)$.

Now we deal with point (1).

\begin{lemma}
For any $\mathbb{Z}_3$-action on a positive definite $4$-manifold $X$ with $H_1(X ; \mathbb{Z}) = 0$ we have that $\langle w_a w_b w_c \rangle =0$ for distinct $a,b,c$.
\end{lemma}
\begin{proof}
As in Section \ref{sec:definite}, we consider an equivariant spin$^c$-structure $\mathfrak{s}$ which has $c(\mathfrak{s}) = \epsilon_1 \widehat{e}_1 + \cdots + \epsilon_t \widehat{e}_t + \widehat{f}$. Then the index $Spin \in R[G]$ of the corresponding Dirac operator is trivial, giving

\begin{align}
0 &= \sum_i \frac{ \omega^{w_i} }{(\omega^{a_i} - \omega^{-a_i})(\omega^{b_i} - \omega^{-b_i})} \label{equ:spin3} \\
&+ \sum_j \left( \frac{1}{2}\langle c(\mathfrak{s}) , [\Sigma_j] \rangle \frac{ \omega^{v_j} }{(\omega^{c_j} - \omega^{-c_j})} -  \frac{1}{2}[\Sigma_j]^2 \frac{(\omega^{c_j} + \omega^{-c_j}) \omega^{v_j}}{(\omega^{c_j} - \omega^{-c_j})^2} \right). \nonumber
\end{align}

Here $\omega = e^{2\pi i/3}$. Set $\widehat{\delta}_j = [\Sigma_j]^2$, then $\widehat{\delta}_j$ is a lift of $\delta_j$ to $\mathbb{Z}$ and set $(\widehat{u}_a)_j = \langle e_a , [\Sigma_j] \rangle$, then $(\widehat{u}_a)_j$ is a lift of $(u_a)_j$ to $\mathbb{Z}$. We also set $\widehat{u}(\epsilon)_j = \sum_a \epsilon_a (\widehat{u}_a)_j$.

Using $1/(\omega - \omega^{-1}) = (\omega^2-\omega)/3$, Equation (\ref{equ:spin3}) simplifies to
\begin{equation}\label{equ:spin3-2}
\sum_i 2q_i \omega^{w(\epsilon)_i} + \sum_j \widehat{u}(\epsilon)_j \omega^{v(\epsilon)_j} (\omega^{c_j} - \omega^{-c_j}) + \sum_j \widehat{\delta}_j \omega^{v(\epsilon)_j} = 0
\end{equation}
where $q_i = 1$ if $a_i = b_i$, $q_i = -1$ if $a_i \neq b_i$, $w(\epsilon)_i = \sum_a \epsilon_a (w_a)_i$, $v(\epsilon)_j = \sum_a \epsilon_a (v_a)_j$. Let $\widehat{c}_j \in \mathbb{Z}$ be the unique lift of $c_j$ to $\mathbb{Z}$ taking values in $\{ 1, -1\}$. Then $\omega^{c_j} - \omega^{-c_j} = \widehat{c}_j(\omega - \omega^2)$, so Equation (\ref{equ:spin3-2}) simplifies to
\begin{equation}\label{equ:spin3-3}
\sum_i 2q_i \omega^{w(\epsilon)_i} + \sum_j \widehat{u}(\epsilon)_j \omega^{v(\epsilon)_j} \widehat{c}_j (\omega - \omega^{2}) + \sum_j \widehat{\delta}_j \omega^{v(\epsilon)_j} = 0.
\end{equation}
This is an equation in $\mathbb{Z}[\omega]$. We consider the reduction of Equation (\ref{equ:spin3-3}) to $\mathbb{Z}_9[\omega]$. Let $f : \mathbb{Z} \to \mathbb{Z}_9$ be given by $f(c) = c^3 + 3c^2 - 2$. It is easily checked that $f(0) = -2$, $f(1) = 2$, $f(2) = 0$ and $f(c+3) = f(c)$. It then follows that for any $c \in \mathbb{Z}$, we have
\[
2\omega^c = f(c) \omega + f(-c)\omega^2 \text{ in } \mathbb{Z}_9[\omega].
\]
Consider the basis of $\mathbb{Z}_9[\omega]$ over $\mathbb{Z}_9$ given by $e = \omega + \omega^2$, $f = \omega - \omega^2$. Then we find using the above that
\[
2\omega^c = (3c^2 - 2)e + c^3 f \text{ in } \mathbb{Z}_9[\omega].
\]
Substituting into Equation (\ref{equ:spin3-3}) and noting that $ef = -f$, $f^2 = 3e$, we obtain:
\begin{align*}
& 2 \sum_i q_i \left( (3 w(\epsilon)_i^2 - 2)e + w(\epsilon)_i^3 f \right) + \sum_j \widehat{u}(\epsilon)_j \widehat{c}_j \left( (2-3v(\epsilon)_j^2)f + 3 v(\epsilon)_j^3 e \right) \\
& + \sum_j \widehat{\delta}_j \left( (3 v(\epsilon)_j^2 - 2)e + v(\epsilon)_j^3 f \right) = 0.
\end{align*}
Equating the coefficients of $f$ to zero, we have
\[
2 \sum_i q_i w(\epsilon)_i^3 + \sum_j \widehat{u}(\epsilon)_j \widehat{c}_j (2-3v(\epsilon)_j^2) + \sum_j \widehat{\delta}_j v(\epsilon)_j^3 = 0 \text{ in } \mathbb{Z}_9.
\]
Let $a,b,c$ be distinct. Expanding and collecting $\epsilon_a \epsilon_b \epsilon_c$ terms, we get
\begin{align*}
0&= 12 \sum_i q_i (w_a)_i (w_b)_i (w_c)_i  \\
& - 6 \sum_j \widehat{c}_j \left( (\widehat{u}_a)_j (v_b)_j (v_c)_j + (\widehat{u}_b)_j (v_a)_j (v_c)_j + (\widehat{u}_c)_j (v_a)_j (v_b)_j \right) \\
& + 6 \sum_j \widehat{\delta_j} (v_a)_j (v_b)_j (v_c)_j \text{ in } \mathbb{Z}_9.
\end{align*}

Hence after cancelling factors of $3$, we get the following equation in $\mathbb{Z}_3$:
\begin{align*}
0 &= \sum_i q_i (w_a)_i (w_b)_i (w_c)_i  \\
& + \sum_j c_j \left( (u_a)_j (v_b)_j (v_c)_j + (u_b)_j (v_a)_j (v_c)_j + (u_c)_j (v_a)_j (v_b)_j \right) \\
& - \sum_j \delta_j (v_a)_j (v_b)_j (v_c)_j \text{ in } \mathbb{Z}_3.
\end{align*}

Noting that $q_i = 1/(a_i b_i)$, $c_j = 1/c_j$ and $1/c_j^2 = 1$ in $\mathbb{Z}_3$, the above equation says that $\langle w_a w_b w_c \rangle = 0$.

\end{proof}

In the case $p=3$ we amend the definition of a weight system to include an addditional axiom:
\begin{itemize}
\item[(W9)]{$\langle w_a w_b w_c \rangle = 0$ for all distinct $a,b,c$.}
\end{itemize}

When $p \neq 2,3$, $(W9)$ follows from $(W6)$ and so does not need to be included in the definition. With this amended definition of a weight system and taking into account points (1) and (2) above, we see that the proof of Theorem \ref{thm:wtdec} carries over to the case $p=3$. When decomposing $W$ into a connected sum of weight systems $W_1$, $W_2$, we now also need to check that $W_1,W_2$ satisfy $(W9)$, but this is trivial.

\subsection{Case $p=2$}

Consider a $G = \mathbb{Z}_2 = \langle g \rangle$-action on a positive definite $4$-manifold $X$ with $H_1(X ; \mathbb{Z}) = 0$. The proof of Proposition \ref{prop:decomp} does not carry over to the case $p=2$ because $H^2(X ; \mathbb{Z})$ could also contain cyclotomic summands, that is, elements $e \in H^2(X ; \mathbb{Z})$ with $e^2 = 1$ and $ge = -e$. If $H^2(X ; \mathbb{Z})$ contains cyclotomic summands then there are no $G$-invariant spin$^c$-structures and our methods break down. However, if we assume that $H^2(X ; \mathbb{Z})$ contains no cyclotomic summands then the results of Section \ref{sec:definite} and \ref{sec:decomp} do carry over, except at a few specific points that we will address below.

For the rest of this section we assume that $H^2(X ; \mathbb{Z})$ contains no cyclotomic summands. In this case the localisation theorm implies that the fixed point set of $g$ consists of isolated points and embedded spheres, just as we had for odd $p$. The fixed point data $(a,b,c,\delta,w)$ satisfies conditions $(W1)$-$(W5)$. Axioms $(W6)$-$(W8)$ will not be needed, but we will need two additional axioms:
\begin{itemize}
\item[(W9)]{$\langle w_a w_b w_c \rangle = 0$ for all $a,b,c$ distinct.}
\item[(W10)]{$\sum_a (u_a)_j = \delta_j$ for all $j$, where $w_a = ( \{ (w_a)_i \} ; \{ (v_a)_j + (u_a)_j \tau \} )$.}
\end{itemize}

We will say that $W = (a,b,c,\delta,w)$ is a {\em weight system} for $p=2$ if it satisfied $(W1)$-$(W5)$ as well as $(W9)$ and $(W10)$.

Suppose that $W = (a,b,c,\delta,w)$ comes from a $\mathbb{Z}_2$-action on a positive definite $4$-manifold $X$ with $H_1(X ; \mathbb{Z}) = 0$ and $H^2(X ; \mathbb{Z})$ having no cyclotomic summands. Then clearly $W$ satisfies $(W1)$-$(W5)$. We will show that $(W9)$ and $(W10)$ also hold.

\begin{lemma}\label{lem:delta2}
For $p=2$, we have $\sum_a (u_a)_j = \delta_j$.
\end{lemma}
\begin{proof}
We have $[\Sigma_j] = \sum_a \langle e_a , [\Sigma_j] \rangle e_a + \sum_k r_k f_k$ for some $r_k \in \mathbb{Z}$. Noting that $\langle e_a , f_k \rangle = 0$ and $\langle f_k , f_{k'} \rangle = 0 \; ({\rm mod} \; 2)$, it follows that 
\[
\delta_j = [\Sigma_j]^2 = \sum_a \langle e_a , [\Sigma_j] \rangle = \sum_a (u_a)_j \; ({\rm mod} \; 2).
\]
\end{proof}

\begin{lemma}\label{lem:abcp2}
We have $\langle w_a w_b w_c \rangle = 0$ for all distinct $a,b,c$.
\end{lemma}
\begin{proof}
In Proposition \ref{prop:lef} we computed $\chi_{Spin}(g^2)$ in the case $p$ is odd. In the case $p=2$, a similar calculation gives
\[
\chi_{Spin}(g) = \frac{1}{4}\sum_i \chi(\mathfrak{s},x_i) + \frac{1}{4}\sum_j \chi(\mathfrak{s},\Sigma_j) \langle c(\mathfrak{s}) , [\Sigma_j] \rangle
\]
where $\chi(\mathfrak{s},x_i), \chi(\mathfrak{s},\Sigma_j) \in \{ 1,-1\}$ are defined as follows. Let $S^{\pm}$ be the spinor bundles corresponding to $\mathfrak{s}$. At the isolated point $x_i$ with weights $(a_i,b_i)$ we have $S^{\pm}|_{x_i} = E_i \otimes S^{\pm}_{can}$ for some $1$-dimensional representation $E_i$ of $\mathbb{Z}_2$. Then $\chi(\mathfrak{s},x_i)$ denotes the character of $E_i$ evaluated on $g$. Similarly on the surface $\Sigma_j$ we have $S^{\pm}|_{\Sigma_j} = F_j \otimes S^{\pm}_{can}$ for some complex $\mathbb{Z}_2$-equivariant line bundle $F_j$. The fibre of $F_j$ over any point of $\Sigma_j$ is a $1$-dimensional representation of $\mathbb{Z}_2$ and $\chi(\mathfrak{s},\Sigma_j)$ denotes the character of this representation evaluated on $g$.

If we take $\mathfrak{s} = \mathfrak{s}(\epsilon)$ to have $c(\mathfrak{s}) = \epsilon_1 \widehat{e}_1 + \cdots + \epsilon_t \widehat{e}_t + f$, then $\chi_{Spin}(g) = 0$ by \cite[Theorem 8.1]{bar}. Hence we obtain
\begin{equation}\label{equ:indp21}
0 = \sum_i \chi_i(\epsilon) + \sum_j \mu_j(\epsilon) \sum_a \epsilon_a (\widehat{u}_a)_j
\end{equation}
where $\chi_i(\epsilon) = \chi( \mathfrak{s}(\epsilon) , x_i)$, $\chi_j(\epsilon) = \chi( \mathfrak{s}(\epsilon) , \Sigma_j )$.

Let $\theta_a \in \{0,1\}$ be chosen so that $\epsilon_a = (-1)^{\theta_a}$. Set $\chi_i = \chi_i( (1,1,\dots , 1) )$ and $\mu_j = \mu_j( (1,1, \dots , 1))$. It follows that
\[
\chi_i(\epsilon) = \chi_i (-1)^{\theta_1 (w_1)_i + \cdots + \theta_t (w_t)_i }, \quad \mu_j(\epsilon) = \mu_j (-1)^{\theta_1 (v_1)_j + \cdots + \theta_t (v_t)_j }.
\]
Then Equation (\ref{equ:indp21}) can be re-written as
\begin{equation}\label{equ:indp22}
0 = \sum_i \chi_i (-1)^{\theta_1 (w_1)_i + \cdots + \theta_t (w_t)_i } + \sum_j \mu_j (-1)^{\theta_1 (v_1)_j + \cdots + \theta_t (v_t)_j } \sum_a (-1)^{\theta_a} (\widehat{u}_a)_j.
\end{equation}
This is an equation in $\mathbb{Z}$. We reduce to $\mathbb{Z}_{16}$ and make use of the fact that $(-1)^w = 1-2w^4 \; ({\rm mod} \; 16)$ for any $w \in \mathbb{Z}$. Hence Equation (\ref{equ:indp22}) gives:
\begin{align}
&\sum_i \chi_i (1 - 2(\theta_1 (\widehat{w}_1)_i + \cdots + \theta_t (\widehat{w}_t)_i)^4 ) \label{equ:indp23} \\
& + \sum_j \mu_j (1 - 2(\theta_1 (\widehat{v}_1)_j + \cdots + \theta_t (\widehat{v}_t)_j )^4 )\sum_a (1-2\theta_a)(\widehat{u}_a)_j = 0 \; ({\rm mod} \; 16) \nonumber
\end{align}
where $(\widehat{w}_a)_i, (\widehat{v}_a)_j \in \{0,1\}$ are lifts of $(w_a)_i, (v_a)_j$ to $\mathbb{Z}$.

Taking $\theta_a = 0$ for all $a$ in (\ref{equ:indp23}) gives
\begin{equation}\label{equ:m1}
\sum_i \chi_i + \sum_j \mu_j \sum_a (\widehat{u}_a)_j = 0 \; ({\rm mod} \; 16).
\end{equation}

Taking $\theta_a = 1$ and $\theta_b = 0$ for all $b \neq a$ in (\ref{equ:indp23}) and making use of (\ref{equ:m1}), we obtain:
\begin{equation}\label{equ:m2}
\sum_i \chi_i (\widehat{w}_a)_i + \mu_j (\widehat{v}_a)_j ( \sum_c (\widehat{u}_c)_j - 2(\widehat{u}_a)_j) + \sum_j \mu_j (\widehat{u}_a)_j = 0 \; ({\rm mod} \; 8).
\end{equation}

Taking $\theta_a = \theta_b = 1$, $\theta_c = 0$ for all $c \neq a,b$ in (\ref{equ:indp23}) and making use of (\ref{equ:m1}), (\ref{equ:m2}), we obtain:
\begin{align}
&\sum_i \chi_i (\widehat{w}_a)_i (\widehat{w}_b)_i + \sum_j \mu_j ( (\widehat{v}_a)_j (\widehat{u}_b)_j  + (\widehat{v}_b)_j (\widehat{u}_a)_j ) \label{equ:m3} \\
& + \sum_j \mu_j (\widehat{v}_a)_j (\widehat{v}_b)_j ( \sum_c (\widehat{u}_c)_j - 2(\widehat{u}_a)_j - 2(\widehat{u}_b)_j ) = 0 \; ({\rm mod} \; 4). \nonumber
\end{align}

Taking $\theta_a = \theta_b = \theta_c = 1$, $\theta_d = 0$ for all $d \neq a,b,c$ in (\ref{equ:indp23}) and making use of (\ref{equ:m1}), (\ref{equ:m2}), (\ref{equ:m3}), we obtain:
\begin{align}
&\sum_i \chi_i (\widehat{w}_a)_i (\widehat{w}_b)_i (\widehat{w}_c )_i \label{equ:m4} \\
&+ \sum_j \mu_j ( (\widehat{v}_a)_j (\widehat{v}_b)_j (\widehat{u}_c)_j  + (\widehat{v}_a)_j (\widehat{v}_c)_j (\widehat{u}_b)_j  + (\widehat{v}_b)_j(\widehat{v}_c)_j(\widehat{u}_a)_j ) \nonumber \\
& + \sum_j \mu_j (\widehat{v}_a)_j (\widehat{v}_b)_j (\widehat{v}_c)_j ( \sum_d (\widehat{u}_d)_j - 2(\widehat{u}_a)_j - 2(\widehat{u}_b)_j -2(\widehat{u}_c)_j ) = 0 \; ({\rm mod} \; 2). \nonumber
\end{align}
But since $\chi_i = \mu_j = 1 \; ({\rm mod} \; 2)$ and since $\sum_d (\widehat{u}_d)_j = \delta_j \; ({\rm mod} \; 2)$ by Lemma \ref{lem:delta2}, this reduces to

\begin{align*}
&\sum_i (w_a)_i (w_b)_i (w_c )_i \\
&+ \sum_j ( (v_a)_j (v_b)_j (u_c)_j  + (v_a)_j (v_c)_j (u_b)_j  + (v_b)_j(v_c)_j(u_a)_j ) \\
& + \sum_j \delta_j (v_a)_j (v_b)_j (v_c)_j  = 0 \; ({\rm mod} \; 2),
\end{align*}
which says that $\langle w_a w_b w_c \rangle = 0$.

\end{proof}

\begin{lemma}\label{lem:surf}
Let $W = (a,b,c,\delta,w)$ satisfy $(W1)-(W5)$ for $p=2$. Suppose $W$ is of type $(n,m)$. Then $m > 0$.
\end{lemma}
\begin{proof}
Suppose that $m=0$. Then for any $a$, $w_a = ( (w_a)_1 , \dots , (w_a)_n )$ for some $(w_a)_i \in \mathbb{Z}_2$. But this implies $w_a^2 = w_a$ and hence $0 = \langle w_a \rangle = \langle w_a^2 \rangle = 1$, a contradiction.
\end{proof}

By Lemma \ref{lem:surf}, we can always weight shift so that $W$ is based at a surface.

\begin{lemma}\label{lem:dabdba2}
Let $W = (a,b,c,\delta,w)$ be a weight system which comes from a $\mathbb{Z}_2$-action on a positive definite $4$-manifold $X$ with $H_1(X ; \mathbb{Z}) = 0$ and the action of $\mathbb{Z}_2$ having no cyclotomic summands. Assume that $W$ is based at a surface. Then $D_{ab}D_{ba} = 0$.
\end{lemma}
\begin{proof}
Using Lemma \ref{lem:abcp2}, the proof of Proposition \ref{prop:products} carries over to the $p=2$ case. Thus $w_a w_b = D_{ab}w_a + D_{ba}w_b$ for any $a \neq b$.

Writing $w_a = ( \{ (w_a)_i \} ; \{ (v_a)_i + (u_a)_i \tau \} )$, we see that $w_a^2 + w_a = ( 0 ; \{ (u_a)_i \tau \} )$. Hence for any $a \neq b$, we have $(w_a^2 + w_a)(w_b^2 + w_b) = 0$. Expanding gives:
\begin{equation}\label{equ:wab0}
w_a^2 w_b^2 + w_a^2 w_b + w_b^2 w_a + w_a w_b = 0.
\end{equation}
From $w_a w_b = D_{ab}w_a + D_{ba} w_b$, we get $w_a^2 w_b^2 = (D_{ab}w_a + D_{ba}w_b)^2 = D_{ab}w_a^2 + D_{ba}w_b^2$ (using $x^2 = x$ for any $x \in \mathbb{Z}_2$). Substituting into Equation (\ref{equ:wab0}) gives
\begin{align*}
0 &= D_{ab}(w_a^2+w_a) + D_{ba}(w_b^2 + w_b) + (w_a+w_b)(D_{ab}w_a + D_{ba}w_b) \\
&= D_{ab}w_a + D_{ba}w_b + D_{ab}w_a w_b + D_{ba}w_a w_b \\
&= (D_{ab} + D_{ab}^2 + D_{ab}D_{ba})w_a + (D_{ba} + D_{ba}^2 + D_{ab}D_{ba})w_b \\
&= D_{ab}D_{ba}w_a + D_{ab}D_{ba}w_b.
\end{align*}
Hence $D_{ab}D_{ba} = 0$.
\end{proof}

Next we explain how to adapt the proof of Theorem \ref{thm:wtdec} to the case $p=2$. We list below all the places where axioms $(W6),(W7),(W8)$ were used and then explain how the same results hold for $p=2$, assuming $(W9), (W10)$.

\begin{itemize}
\item[(1)]{We used $(W6),(W7)$ in the proof of Proposition \ref{prop:t1}.}
\item[(2)]{We used $(W6)$ to deduce that $\langle w_a w_b w_c \rangle = 0$.}
\item[(3)]{We used $(W6),(W7)$ in the proof of Proposition \ref{prop:dabdba}.}
\item[(4)]{We used $(W6),(W7),(W8)$ in Case 1 of the proof of Theorem \ref{thm:wtdec}.}
\item[(5)]{We used $(W6)$ in Cases 2 and 3 of the proof of Theorem \ref{thm:wtdec}.}
\end{itemize}

We address each of these points below.

{\bf Point (1).} Let $W = (a,b,c,\delta,w)$ be a weight system for $p=2$ with $n+2m = 3$. We need to show that $W$ is the weight system of a linear $\mathbb{Z}_2$-action on $\mathbb{CP}^2$. By Lemma \ref{lem:surf}, $W$ is of type $(1,1)$ so we can write is as $W = (a,b,c,\delta,w)$ with $w = (w_0 ; w_2 \tau)$. Since $a,b,c \neq 0$, we have $a=b=c=1$. From $0 = \langle w \rangle = w_0 + w_2$, we have $w_0 = w_2$ and from $1 = \langle w^2 \rangle = w_0$, we get $w_0 = w_2 = 1$. Lastly from $(W10)$ we get $\delta = w_2 = 1$. This is the weight system for the $\mathbb{Z}_2$-action on $\mathbb{CP}^2$ given by $g[x,y,z] = [x,-y,-z]$.

{\bf Point (2).} We proved $\langle w_a w_b w_c \rangle = 0$ in Lemma \ref{lem:abcp2}.

{\bf Point (3).} We proved $D_{ab} D_{ba} = 0$ in Lemma \ref{lem:dabdba2}.

{\bf Point (4).} In Case 1, we have $w_0 = (u_1,u_2, 0 , \dots , 0 ; 0 , \dots , 0)$ where $u_1,u_2 \neq 0$ and $u_1 \neq u_2$. For $p=2$ this is impossible, so Case 1 does not occur.

{bf Point (5).} In Case 2, we used $(W6)$ to deduce that $\delta_1 = 1$. For $p=2$, we will instead deduce this using $(W10)$. Since $w_0 = (0 , \dots , 0 ; 1+\tau , 0 , \dots , 0)$ and $w_0 w_a = 0$ for $a \neq 0$, it follows that $(u_a)_1 = 0$ for $a \neq 0$ and $(u_0)_1 = 1$. Then $(W10)$ gives $\delta_1 = \sum_a (u_a)_1 = 1$. In Case 3, we used $(W6)$ to deduce that $v^2 = 1$. But for $p=2$ this is immediate, since $v \neq 0$.

Lastly, since we are using a modified definition of weight system that includes two new axioms $(W9),(W10)$, we must check that the weight systems $W_1,W_2$ produced in the proof of Theorem \ref{thm:wtdec} also satisfy $(W9),(W10)$. This is straightforward to check.


\bibliographystyle{amsplain}

\end{document}